\documentclass[12pt,a4paper]{amsart}
\usepackage{amsfonts}
\usepackage{amssymb}
\usepackage{bbm}
\usepackage{comment}
\usepackage{enumitem}
\usepackage{dsfont}
\usepackage{amsthm}
\usepackage{amsmath}
\usepackage{commath}
\usepackage{url}
\usepackage{epsfig}
\usepackage{subfigure}
\usepackage{adjustbox}
\usepackage[ruled,vlined,linesnumbered]{algorithm2e}
\usepackage{graphicx}
\usepackage{a4wide}
\usepackage{todonotes}
\usepackage{mathtools}
\mathtoolsset{showonlyrefs}

\makeatletter
\@namedef{subjclassname@2010}{%
  \textup{2010} Mathematics Subject Classification}
\makeatother

\usepackage[T1]{fontenc}
\definecolor{green}{rgb}{0.1,1,0.1}

\newcommand{\todoMZ}[2][]{\todo[color=green!60, #1]{Manos: #2}}

\newtheorem{thm}{Theorem}
\newtheorem{lem}[thm]{Lemma}

\theoremstyle{definition}

\newcommand{\ve}{\varepsilon}

\newtheorem*{xrem}{Remark}

\newcommand{\ls}{\leqslant}
\newcommand{\gs}{\geqslant}

\newcommand{\one}{\mathds{1}}
\newcommand{\xn}{(x_n)_{n \in \mathbb{N}} }
\newcommand{\yn}{(y_n)_{n \in \mathbb{N}} }
\newcommand{\zn}{(z_n)_{n \in \mathbb{N}} }

\def \bN {\mathbb N}

\def \bR {\mathbb R}

\def \bs {\mathbf{s}}
\def \bt {\mathbf{t}}
\def \ba {\mathbf{a}}
\def \bb {\mathbf{b}}

\def \bj {\mathbf{j}}
\def \bc {\mathbf{c}}
\newcommand{\cR}{\mathcal{R}}

\newcommand{\fX}{\mathfrak{X}}

\begin{document}
\baselineskip=17pt
\title{Weak Poissonian Correlations }
\author{Manuel Hauke}
\address{TU Graz, Austria}
\email{hauke@math.tugraz.at}
\author[A. Zafeiropoulos]{Agamemnon Zafeiropoulos}
\address{NTNU Trondheim, Norway}
\email{agamemnon.zafeiropoulos@ntnu.no}

\date{}
\thanks{AZ is supported by a postdoctoral fellowship funded by Grant 275113 of the Research Council of Norway} 

\begin{abstract} 
We examine a property of sequences called Poissonian pair correlations with parameter $0\ls \beta \ls 1$ (abbreviated as $\beta$\,--\,PPC). We prove that when $\beta<1,$ the property of $\beta$\,--\,PPC, also known as weak Poissonian correlations, can be detected at the behaviour of sequences at small scales, and show that this does not happen for the classical notion of PPC, that is, when $\beta = 1$. Furthermore, we show that whenever $0\ls \alpha < \beta \ls 1$, $\beta$\,--\,PPC is stronger than $\alpha$\,--\,PPC. We also include a discussion on weak Poissonian correlations of higher orders, showing that for $\beta < 1$, Poissonian $\beta$\,--\,correlations of  order $k+1$ imply  Poissonian $\beta$\,--\,correlations of $k$\,--\,th order with the same parameter $\beta$. 
\end{abstract}

\subjclass[2010]{Primary 11K06, 11J71; Secondary 11K99}
\maketitle

\section{Introduction and statement of results}
\subsection{Introduction} Let $\xn\subseteq [0,1]$ be a sequence of points in the unit interval.  For $s > 0$ and for any positive integer $N \gs 1,$ we define the {\it pair correlation function} of the sequence $\xn$ to be 
\[ R_2(s,N) = \frac{1}{N}\#\Big\{m,n \ls N, m\neq n : \|x_m-x_n\| \ls \frac{s}{N} \Big\},  \]
where $\|x\|$ denotes the distance of $x\in \mathbb{R}$ to the nearest integer (see Section \ref{notation} for a proper definition).
We say that the sequence $\xn$ has {\it Poissonian pair correlations } (from now on abbreviated as PPC) if 
\begin{equation} \label{PPC}
\lim_{N\to \infty} R_2(s,N) = 2s \qquad \text{ for all } s>0.  
\end{equation}
The notion of pair correlations of sequences has been studied in various contexts. Its natural connection with mathematical physics is exhibited by the famous Berry\,--\,Tabor conjecture \cite{bt}. A series of more recent papers have studied pair correlations from a purely theoretical point of view. To mention an example, it was an open problem within the theoretical setup to determine  the relation of Poissonian pair correlations with uniform distribution. It has been recently shown that any sequence $\xn$ with PPC is also uniformly distributed mod $1,$ that is, we have 
\[ \lim_{N\to \infty} \frac{1}{N}\# \{i\ls N : x_i \in [a,b] \} = b-a \quad \text{ for all }\, 0\ls a < b \ls 1. \]
This result was established by Aistleiter, Larcher and Lewko \cite{all} and independently by Grepstad and Larcher \cite{sigrid}, while a subsequent proof was also given by Steinerberger \cite{steinerberger} and, in a much more general setup, by Marklof \cite{marklof}. \par
In the present paper, we focus our attention on the notion of weak Poissonian pair correlations with parameter $0\ls \beta \ls 1$. As we shall soon explain, this notion forms a weaker variant of the classical property of PPC, a fact which explains the term ``weak''. Our main purpose is to compare it with the standard property of Poissonian pair correlations and demonstrate several  differences. To the best of our knowledge, the notion of weak Poissonian pair correlations  was first introduced by Nair and Pollicot in \cite{nair}. 
\par
	Let $N \gs 1, 0 \ls \beta \ls 1$ and $s>0$. We define the \textit{pair correlation function with parameter $\beta$} of a sequence $(x_n)_{n \in \mathbb{N}}\subseteq [0,1]$ to be 
	\begin{equation*}
	R_2(\beta;s,N) = \frac{1}{N^{2 - \beta}} \#\Big\{m,n \ls N, m\neq n : \|x_m-x_n\| \ls \frac{s}{N^{\beta}} \Big\}. 
	\end{equation*}
From a statistical perspective, the variants of the pair correlation function (with a different scaling  factor) which are considered in the present paper fall into the framework of Ripley's K-function \cite{ripley}. 
\par We say that the sequence $(x_n)_{n \in \mathbb{N}}$ has \emph{weak Poissonian pair correlations with parameter $0 < \beta < 1$} (also called \emph{Poissonian $\beta$\,--\,pair correlations}, or, for abbreviation, $\beta$\,--\,PPC) if 
\begin{equation}\label{weak_defi}
\lim_{N \to \infty} R_2(\beta;s,N) = 2s \qquad \text{ for  all }\, s> 0. \end{equation} 
For the value $\beta=0$, we say that the sequence $(x_n)_{n \in \mathbb{N}}$ has $0$\,--\,PPC if
\begin{equation*}
\lim_{N \to \infty} R_2(\beta;s,N) = 2s \qquad \text{ for  all }\, 0 < s \ls \frac{1}{2}\cdot \end{equation*} 
 The reason why in the definition of $0$\,--\,PPC the values of the scale $s$ are restricted to the range $0< s\ls \frac12$ is quite simple: when $\beta = 0$ we trivially have $R_2(0;s,N) \ls 1$ and thus \eqref{weak_defi} cannot hold for $s > \frac{1}{2}$. \par  
It is clear that for $\beta=1$ the notion of $\beta$\,--\,PPC coincides with the classical Poissonian pair correlations property.  For $\beta=0$, the property of $0$\,--\,PPC appears to be known as being equivalent to uniform distribution, see e.g. \cite[Section 2.2]{steinerberger}. Since we have not managed to find a rigorous explanation in the literature, we provide a proof of this fact.
  
\begin{thm} \label{prop1} Let $(x_n)_{n \in \mathbb{N}} \subseteq [0,1]$ be a sequence. The following are equivalent. 
\begin{enumerate} 
\item[(i)] The sequence $(x_n)_{n \in \mathbb{N}}$ is uniformly distributed mod $1$. 
\item[(ii)] The sequence $\xn$ has $0$\,--\,PPC.
\end{enumerate}
\end{thm} 

We now turn our attention to a rather striking aspect of weak Poissonian correlations which, in our opinion, constitutes the most remarkable difference of $\beta$\,--\,PPC on the one hand, and the classical property of PPC, on the other. We show that in order for a given sequence $\xn$ to have $\beta$\,--\,PPC for some $0\ls \beta <1$, it is sufficient that the defining condition $\lim\limits_{N \to \infty} R_2(\beta;s,N) = 2s$ holds for all values of $s$ that lie within an interval of the form $(0,s_0),$ where $s_0>0$ can be chosen to be arbitrarily small. 
\begin{thm}\label{small_scales_suffice}
Let $0 \ls \beta < 1$.  Assume there exists some constant $s_0>0$ such that 
\begin{equation}
\lim_{N \to \infty}R_2(\beta;s,N) = 2s \qquad \text{ for all } s < s_0.\end{equation}
Then the sequence $\xn$ has $\beta$\,--\,PPC.
\end{thm}
 The key that served as a hint for this surprising fact was the proof of Theorem \ref{prop1}. As the reader will be able to see, to deduce uniform distribution from the hypothesis of 0\,--\,PPC in the proof of Theorem \ref{prop1}, we only employ the fact that $R_2(0;s,N)\to 2s$ for all values of $s>0$ that are sufficiently small. \par As already alluded to, the assumption that $\beta$ is strictly less than $1$ in Theorem \ref{small_scales_suffice} turns out to be essential. In other words, it is not possible to extend Theorem \ref{small_scales_suffice} to the setup of PPC. We actually prove a rather stronger statement: for any choice of the number $S>0$, we can find a sequence $\xn$ with $R_2(s,N)\to 2s$ for all $s<S$ which is not even uniformly distributed, whence it does not have PPC.

\begin{thm} \label{stronger_not_for_one}
For every $S > 0$, there exists a sequence $\xn$ such that 
     \[ \lim_{N\to\infty} R_2(s,N) = 2s \qquad \text{ for all } s<S \]
     but the sequence $\xn$ is not uniformly distributed, and hence does not have PPC.
\end{thm}
The statement of Theorem \ref{stronger_not_for_one} is in agreement with a phenomenon that  enthusiasts of the theory of Poissonian correlations might have observed: all existing proofs of the fact that PPC implies uniform distribution \cite{alp,sigrid,steinerberger} at some point use the fact that the convergence $R_2(s,N) \to 2s$ holds for arbitrarily large values of $s>0.$\par In a remark after the proof of Theorem \ref{prop1}, we give a short, elementary proof of the fact that sequences with $R_2(s,N)\to 2s$ for all $s>s_0$ are uniformly distributed \mbox{modulo $1$}. This assumption can actually be relaxed: in order for a sequence to be uniformly distributed mod $1$, it suffices that $R_2(s,N)\to 2s$ holds for all positive integers $s\in \mathbb{N}$. This was proved in \cite{hinrichs} in a multidimensional setup, and later in \cite{steinerberger}, where it was further pointed out that a sequence is equidistributed as long as $R_2(s,N)\to 2s$ holds for all $s$ in a discrete set satisfying some ``maximum gap'' condition. The fact that sequences with $R_2(s,N)\to 2s$ for all $s>s_0$ are uniformly distributed follows directly from this observation in \cite{steinerberger}, but we nevertheless decided to include one more proof because of its simplicity and its underlying connection with the arguments in \cite{alp}.
\par Concerning the relation of weak Poissonian pair correlations with uniform distribution, Steinerberger \cite{steinerberger2, steinerberger} proved that if the sequence $\xn$ has  $\beta$\,--\,PPC for some $0<\beta\ls 1$ then it is uniformly distributed. This showed that the property of $\beta$\,--\,PPC is stronger than $0$\,--\,PPC for any $0<\beta \ls 1,$ extending the already known fact that $1$\,--\,PPC is a stronger property than $0$\,--\,PPC (uniform distribution).
\par  Furthermore, Steinerberger refers in \cite{steinerberger} to  $\beta$\,--\,PPC (for a given $0< \beta < 1)$ as a property that interpolates between uniform distribution and PPC. The following result  describes this phenomenon in a more precise fashion: whenever   $0\ls \alpha < \beta \ls 1,$  the property of $\beta$\,--\,PPC  is stronger than $\alpha$\,--\,PPC.

\begin{thm}\label{beta_implies_alpha}
	Let $0 \ls \alpha < \beta \ls 1.$ If a sequence $\xn\subseteq [0,1]$ has weak Poissonian pair correlations with parameter $\beta$, then it also has weak Poissonian pair correlations with parameter $\alpha$.
\end{thm}

In turn, Theorem \ref{beta_implies_alpha} leads naturally to the following question:  are the notions of $\beta$\,--\,PPC for the various values of $0<\beta<1$  essentially different or do they actually coincide, possibly for certain values of $\beta$? We prove that the former is the case.  
 \begin{thm}\label{beta_threshold}
    For any $0 < \beta < 1$, there exists a sequence $(y_n)_{n\in\bN}$ that has Poissonian $\alpha$-pair correlations for any $0\ls \alpha < \beta$ but does not have Poissonian $\alpha$\,--\,pair correlations for any $\beta \ls \alpha \ls 1$.
 \end{thm}

Therefore  for the different values of the parameter $0\ls \beta\ls 1$ the property of \mbox{$\beta\text{\,--\,PPC}$}  can be seen as covering a spectrum ranging from PPC to uniform distribution, with  $\beta_1$\,--\,PPC being genuinely stronger than $\beta_2$\,--\,PPC whenever $\beta_1 > \beta_2.$ We note that an analogue of Theorem \ref{beta_threshold} is already known for the value $\beta=1$: the van der Corput sequence has $\alpha$\,--\,PPC for any $\alpha<1$ but does not have PPC, see e.g. \cite{sw}. The same paper \cite{sw} also contains an alternative proof of Theorem \ref{beta_implies_alpha}; in our opinion, though, this proof  contains an oversight  since the range $s$ is allowed to vary depending on $N$.  

\subsection{A metric consideration} Another difference of the weak Poissonian correlations compared to the standard PPC occurs in the metric setup. 
From the metric point of view, an increasing sequence $(a_n)_{n\in\mathbb{N}}$ of positive integers is considered fixed, and we examine the Lebesgue measure of the set of those $x\in [0,1]$ for which the sequence $(a_n x)_{n\in\mathbb{N}}$ has Poissonian pair correlations. Towards this direction, several results have been proved for specific choices of the sequence $(a_n)_{n\in\mathbb{N}}$. To name a few examples, we mention that for any exponent $k\gs 2$ the sequence $(n^k x)_{n\in\mathbb{N}}$   has Poissonian pair correlations for almost all $ x\in [0,1]$, see \cite{heathbrown,rs}. On the other hand, writing $p_n$ for the $n$-th prime number, the sequence $(p_n x)_{n\in\mathbb{N}}$ does not have Poissonian pair correlations for almost all $x\in [0,1]$, see \cite{walker}. For more such results we refer to \cite{all,rz,rz2}.  \par
A fundamental question in the theory of metric Poissonian pair correlations is whether a zero\,--\,one law holds in the setup described above. That is, given an increasing sequence $(a_n)_{n\in\mathbb{N}}\subseteq \bN$, does the set of $x\in [0,1]$ such that $(a_n x)_{n\in\mathbb{N}}$ has PPC have Lebesgue measure either $0$ or $1$? Although all results so far suggest that the answer is positive, this question still remains unanswered.\par  We now briefly examine $\beta$\,--\,pair correlations from the metric point of view. Since $\beta-$PPC is equivalent to uniform distribution,  for the $0$\,--\,PPC a zero-one law is true in a trivial sense: for any choice of the sequence $(a_n)_{n\in\mathbb{N}}$ the sequence $(a_n x)$ is uniformly distributed and hence has $0-$PPC for almost all $x$ (see \mbox{\cite[Theorem 4.1]{kuipers}}). We show that this is also the case with weak pair correlations for all ranges $0<\beta<1.$ \par In view of the multiple examples of sequences $(a_n)_{n\in\mathbb{N}}\subseteq \bN$ for which $(a_n x)_{n\in\mathbb{N}}$ fails to have PPC almost surely (see references above), the following theorem exhibits another difference between the properties of weak Poissonian correlations with the standard notion of PPC.  
\begin{thm} \label{weak01law}
Let $(a_n)_{n\in\bN}$ be an increasing sequence of positive integers and \mbox{$0< \beta <1.$} Then for almost all $x\in [0,1]$ the sequence $(a_n x)_{n\in\bN}$ has Poissonian $\beta$\,--\,pair correlations.   
\end{thm}

\subsection{Weak correlations of higher orders} As a last part of this paper, we extend the previous discussion on weaker variants of Poissonian correlations to orders greater than $2$. We shall study the $k$\,--\,th order correlations of sequences rescaled by a factor of $N^\beta,$ which for convenience we name $(k,\beta)$\,--\,correlations. Given any integer $k \gs 2$, a parameter $0\ls \beta \ls 1$ and a closed rectangle 
\begin{equation}\label{form_of_rect}\mathcal{R} = [a_1,b_1]\times [a_2,b_2] \times \ldots \times [a_{k-1},b_{k-1}] \subseteq \mathbb{R}^{k-1},\end{equation} we define the $(k,\beta)$-correlation function of a sequence $\xn$ as
  \begin{align*}
&R_{k}(\beta;\mathcal{R},N)\! =\! \frac{1}{N^{k - (k-1)\beta}}
 \#\!\!\left\{ \parbox{7em}{ $i_1,\ldots,i_k\ls N$ \\ \text{ \, distinct }} \hspace{-5mm} : N^{\beta} ( (\!(x_{i_1}- x_{i_2})\!),\ldots, (\!(x_{i_1} - x_{i_{k-1}})\!) ) \in \mathcal{R} \right\}\!.
 \end{align*}
(Here and in what follows, we say that the indices $i_1,\ldots,i_k$ are distinct when $i_m\neq i_n$ for $m\neq n.$) As one might expect, when $0< \beta \ls 1$ we define a sequence $(x_n)_{n \in \mathbb{N}}$ to have Poissonian $(k,\beta)$\,--\,correlations if for all rectangles $ \mathcal{R}\subseteq\mathbb{R}^{k-1}$  as in \eqref{form_of_rect} 
it holds that
 \begin{equation*} \label{k_beta_def} \lim_{N \to \infty} R_{k}(\beta; \mathcal{R},N) = \lambda(\mathcal{R}), \end{equation*}
 where $\lambda$ denotes the $(k-1)$\,--\,dimensional Lebesgue measure.  Evidently, when $\beta=1$ in the preceding definition, one obtains the usual definition of Poissonian $k$\,--\,th order correlations \cite{kr}. \par For the specific value $\beta=0$, we say that a sequence has Poissonian $(k,0)$\,--\,correlations if $\lim_{N\to\infty}R_k(0;\cR,N) =\lambda(\mathcal{R})$
holds for any rectangle $\mathcal{R} \subseteq [-\frac12, \frac12]^{k-1}.$ \par We note that the previous definitions are generalizations of just one of the equivalent ways to define  Poissonian $k$\,--\,th order correlations.  We refer the reader to \mbox{\cite[Appendix A]{our_other_paper}} for a relevant discussion, which can be easily adapted to the context of weak correlations.
 \par As a first result in this section, we prove that for any order $k\gs 2$, Poissonian $(k,0)$\,--\,correlations are equivalent to the property of uniform distribution, thus generalising Theorem \ref{prop1}.
\begin{thm} \label{thm7} Let $k \gs 2$ be an integer. Then a sequence $\xn \subseteq [0,1]$ is uniformly distributed if and only if it has Poissonian $(k,0)$\,--\,correlations. 
\end{thm}
As a corollary, Theorem \ref{thm7} implies that Poissonian $(k+1,0)$\,--\,correlations are equivalent to Poissonian $(k,0)$\,--\,correlations for any $k \gs 2$. This brings us to the main motivating point for studying weak correlations of higher orders. \par It is an interesting open problem  to determine whether the property of Poissonian correlations of some given order $k\gs 3$ is a property stronger than Poissonian correlations of lower orders. In the opposite direction, Lutsko, Sourmelidis and Technau recently \cite{lst} gave explicit examples of sequences with Poissonian pair correlations, but without Poissonian triple correlations. Their argument can be easily modified to yield the existence of sequences with Poissonian pair correlations but without Poissonian $k$\,--\,th order correlations for $k>3.$ In addition, the sequence $(\sqrt{n})_{n^2 \notin \mathbb{N}}$ is known to have Poissonian pair correlations \cite{elbaz} but does not have Poissonian correlations of $k$\,--\,th order for at least one value of $k\geqslant 3$; if the correlations of all orders followed the Poisson model, the gap distribution of this sequence would be the Poisson distribution \cite[Appendix]{kr}. However, this is not the case for $(\sqrt{n})_{n^2 \notin \mathbb{N}}$, as was proved by Elkies and McMullen in \cite{elkies}. \par 
It still remains unknown whether the other implication holds, i.e. whether Poissonian $k+1$\,--\,correlations imply Poissonian $k$\,--\,correlations. In the setup of weak  correlations, it turns out that we are able to answer such a question positively.

 \begin{thm}\label{higher_order_weak}
Let $k > 2, 0 \ls \beta < 1$. If a sequence $(x_n)_{n \in \mathbb{N}}$ has Poisonnian $(k,\beta)$\,--\,correlations,	then it also has  Poissonian $(k-1,\beta)$\,--\,correlations.  In particular, Poissonian ($k,\beta$)\,--\,correlations imply weak Poissonian pair correlations with parameter $\alpha \ls \beta$ and therefore, uniform distribution. \end{thm}
We end the introductory part of the paper with some directions for further research. Concerning the relation between the properties of Poissonian correlations of different orders, in view of Theorem \ref{higher_order_weak} one might expect that any sequence with Poissonian $(k+1)$\,--\,order correlations also has Poissonian correlations of order $k.$ However, we are hesitant to conjecture that this is indeed the case. Attempting to adapt the proof of Theorem \ref{higher_order_weak} to the case $\beta=1,$ one can notice that the impasse arises from inequality \eqref{compare_R_k*} appearing therein. More precisely, when $\beta=1$ the terms in \eqref{compare_R_k*} that involve correlation functions of orders $2\ls m \ls k-1$ are not negligible when $N\to \infty,$ and this in turn does not allow for a characterisation of Poissonian correlations of order $k$ in terms of the functions $R_k^*$ introduced later in the paper. \par  Further, we are confident that $(k,\beta)$\,--\,Poissonian correlations of any order $k\gs 3$ is a property that can be detected at small scales, in other words an analogue of Theorem \ref{small_scales_suffice} holds for weak correlations of higher
 orders. We believe that in order to prove this fact it would suffice to establish a condition equivalent to Poissonian $(k,\beta)$\,--\,correlations in terms of the function $F_\beta$ defined later in \eqref{Fdef}. 
\subsection{Notation}\label{notation} 
Given two functions $f,g\!:\!(0,\infty)\to \mathbb{R},$ we shall write \mbox{$f(t) = \mathcal{O}(g(t))$},  $t\to\infty$ and $f(t)= o(g(t)),\, t\to\infty$   when   
\[\limsup_{t\to\infty} \frac{|f(t)|}{|g(t)|} < \infty  \quad \text{ or }  \quad \lim_{t\to\infty} \frac{f(t)}{g(t)} =0   \] respectively.  For a real number $x\in \bR,$ we write $\{x\}$ for the fractional part of $x$, $\|x\|=\min\{|x-k|: k\in\mathbb{N}\}$ for the distance of $x$ from its nearest integer, and $$(\!(x)\!)=\begin{cases} \{x\}, &\text{ if } 0\ls x \ls \tfrac{1}{2} \\ 1-\{x\}, &\text{ if } \tfrac{1}{2}< \{x\} < 1  \end{cases} $$ for the signed distance of $x$ from the origin modulo $1$. Further, we use the symbol $\{\,\cdot\,\}^+$ for the function 
\[ \{x\}^+ = \begin{cases}x, &\text{ if } x\gs 0 \\ 0, &\text{ if } x<0 . \end{cases} \] We use the standard notation $e(x)=e^{2
\pi ix}$ and also write $B(x_0,r)= \{x\in [0,1] : \|x- x_0\| \ls r\}$ for the closed interval with center $x_0$ and length $2r$ modulo $1$. Given a positive integer $m\gs 1$ we write $[m]=\{1,\ldots, m\}$ for the set of positive integers which are less than or equal to $m.$

\section{Preliminary Results}
   
For any value of $0\ls \beta\ls 1$ we define the function
\begin{equation} \label{Fdef}
F_{\beta}(t,s,N)  = \frac{1}{N^{1-\beta} }\#\Big\{1 \ls n \ls N: \norm{x_n - t} \ls \frac{s}{2N^{\beta}}\Big\}, \quad 0\ls t \ls 1.   
\end{equation}
Heuristically, $F_\beta$ can be thought of as counting the number of points among the first $N$ terms of $\xn$ lying within an interval of length $s/N^{\beta}$ and random center $0\ls t\ls 1,$ uniformly distributed in $[0,1]$. 
We also define the integral 
\begin{equation} \label{Idef}
    I_\beta(s,N) = \int_{0}^1 F_\beta(t,s,N)^2 \,\mathrm{d}t.
\end{equation}
 The importance of $I_\beta(s,N)$ in the proofs of our main results is clear from the following lemma, which is an analogue of \cite[Proposition 9]{our_other_paper} in the context of weak correlations.

\begin{lem}\label{integral_equivalence}
Let $(x_n)_{n \in \mathbb{N}} \subseteq [0,1]$ be a sequence, $0\ls \beta \ls 1$ and $s_0>0$ be a constant. The following are equivalent. \newline   
{\it (i)} The $\beta$\,--\,pair correlation function satisfies 
\[ \lim_{N\to \infty} R_2(\beta;s,N) = 2s \qquad \text{ for all } s<s_0. \]
{\it (ii)} We have 
\begin{equation*} \label{intermediate}
 \lim_{N \to \infty} \frac{1}{s}\int_{0}^s \!R_2(\beta;\sigma,N) \,\mathrm{d}\sigma  =  s \qquad \text{ for all } s < s_0.
\end{equation*}  
{\it (iii)} The integral $I_\beta(s,N)$ defined in \eqref{Idef} satisfies
	\[\lim_{N \to \infty} I_{\beta}(s,N) = \begin{cases} s^2,    &\text{ if } \beta < 1\\
	s^2 +s,  &\text{ if } \beta = 1\end{cases}
	\qquad \text{ for all } s<s_0.\]
\end{lem}
\begin{proof}
We first prove the equivalence of (i) with (ii). 
Assuming (i) is true, let $s<s_0.$ Then by the monotonicity of $R_2(\beta;\sigma,N)$ as a function of $\sigma$, for any integer $M\gs 1$ and for any $0\ls j \ls M-1$ we get 
\[ \frac{s}{M}R_2\Big(\beta; \frac{js}{M}, N\Big) \ls \int_{[\frac{js}{M},\frac{(j+1)s}{M} ]}R_2\big(\beta;\sigma, N\big) \,\mathrm{d}\sigma\ls \frac{s}{M}R_2\Big( \beta; \frac{(j+1)s}{M}, N\Big).  \]
Summing over all $0\ls j \ls M-1$ we get
\[  \frac{s}{M}\sum_{0\ls j \ls M-1}\! R_2\Big(\beta; \frac{js}{M}, N\Big) \ls \int_{0}^s R_2(\beta;\sigma, N ) \,\mathrm{d}\sigma \ls \frac{s}{M}\sum_{1\ls j \ls M} R_2\Big( \beta; \frac{js}{M}, N\Big). \]
Letting $N\to \infty$ and using the hypothesis in (i) we deduce that (ii) holds. \par Conversely, when (ii) is true, using again the monotonicity of $R(\beta;\sigma,N)$ in $\sigma$ gives that 
\[\frac{1}{\varepsilon}\int_{s-\ve}^s R_2(\beta;\sigma,N)\,\mathrm{d}\sigma \ls R_2(\beta; s,N) \ls \frac{1}{\varepsilon}\int_{s}^{s+\varepsilon} R_2(\beta;\sigma,N)\,\mathrm{d}\sigma \]
for $s<s_0$ fixed and for all $\varepsilon>0$ sufficiently small. Letting first $N\to \infty$ and then $\varepsilon\to 0$ we conclude that $\lim\limits_{N\to\infty}R_2(\beta;s,N)=2s.$ \newline
We now proceed to prove that (ii) is equivalent to (iii). Observe that
\begin{align*}
F_\beta(t,s,N)^2 &= 
	\frac{1}{N^{2 - 2\beta}}\sum_{m,n \ls N}
	\mathds{1}_{B\left(x_m,\frac{s}{2N^{\beta}}\right)\cap B\left(x_n,\frac{s}{2N^{\beta}}\right)}(t) .
	\end{align*}
	Denoting by $\lambda$ the 1\,--\,dimensional Lebesgue measure, we can write
	\begin{align*}
	I_\beta(s,N) &= \frac{1}{N^{2 - 2\beta}}\sum_{m,n \ls N} 
	\lambda\Big(B\Big(x_m,\frac{s}{2N^{\beta}}\Big)\cap B\Big(x_n,\frac{s}{2N^{\beta}}\Big)\Big)  \\&=
	\frac{1}{N^{2 - 2\beta}}\sum_{\substack{m,n \ls N\\m \neq n}}
 	\lambda\Big(B\Big(x_m,\frac{s}{2N^{\beta}}\Big)\cap B\Big(x_n,\frac{s}{2N^{\beta}}\Big)\Big) + \frac{1}{N^{2 - 2\beta}}\cdot \frac{Ns}{N^{\beta}}     \\
	&=  \frac{1}{N^{2 - 2\beta}}\sum_{\substack{m,n \ls N \\m \neq n}}
	\left\{ \frac{s}{N^{\beta}} - \norm{x_n-x_m}\right\}^+  + \frac{s}{N^{1-\beta}}  
	\\&=  \frac{s}{N^{2-\beta}}\sum_{\substack{m,n \ls N\\m \neq n}} 
	\left\{1 - \frac{\norm{x_n-x_m}}{s/N^{\beta}}\right\}^+  +  \frac{s}{N^{1-\beta}}  \,   \cdot
	\end{align*}
On the other hand, 
	\begin{align*}
	\frac{1}{s}\int_{0}^s R_2(\beta;\sigma,N) \,\mathrm{d}\sigma &=
	\frac{1}{s}\int_{0}^s \frac{1}{N^{2 - \beta}}\sum_{\substack{m,n \ls N\\m \neq n}}  \mathds{1}_{B(0,\frac{\sigma}{N^{\beta}})}\left( x_m-x_n\right) \,\mathrm{d}\sigma \\
	&= \frac{1}{sN^{2-\beta}}\sum_{\substack{m,n\ls N\\m \neq n}} \int_{0}^s \mathds{1}_{[\, \norm{x_m-x_n} N^{\beta},\infty)}(\sigma) \,\mathrm{d}\sigma\\&=
	\frac{1}{sN^{2-\beta}}\sum_{\substack{m,n \ls N\\m \neq n}}  
	\left\{s - \frac{\norm{x_n-x_m}}{1/N^{\beta}}\right\}^+\\&=
	\frac{1}{N^{2-\beta}}\sum_{\substack{m,n \ls N\\m \neq n}} 
	\left\{1 - \frac{\norm{x_n-x_m}}{s/N^{\beta}}\right\}^+, 
	\end{align*}
	and this finishes the proof.
\end{proof}

Finally, we mention the following lemma on the size of the integral $I_\beta(s,N)$ defined in \eqref{Idef} that will be used in the proofs of many of the main results.

\begin{lem}\label{integral_lower_bound} Let $0\ls \beta \ls 1$. For any $s>0$ and $N\gs 1$, we have $ I_\beta(s,N)   \gs   s^2.$\end{lem}

\begin{proof}
By the Cauchy\,--\,Schwarz inequality,  
\[ I_\beta(s,N) = \int_{0}^1 F_\beta(t,s,N)^2 \,\mathrm{d}t \gs \left(\int_{0}^1 \!\!F_\beta(t,s,N)\,\mathrm{d}t \right)^2 = s^2. \]
\end{proof}
 
\section{Proof of Theorem \ref{prop1}}

Assume the sequence $(x_n)_{n \in \mathbb{N}}$ is uniformly distributed mod $1$. 
 For any $0 < s \ls \frac{1}{2} $ we have
	\begin{align*}
\left\lvert R_2(0;s,N) - 2s \right\rvert &=  \Big\lvert\frac{1}{N^2}\sum_{\substack{k,\ell \ls N\\k \neq \ell }} \mathds{1}_{[-s,s]}\left( x_k-x_\ell \right) -2s \Big\rvert  \\
&\ls \frac{1}{N}\sum_{k = 1}^N \Big\lvert\frac{1}{N}\sum_{\substack{\ell \ls N \\ \ell \neq k}}\mathds{1}_{[-s,s]}\left(x_k-x_\ell \right) -2s\Big\rvert \\
& = \frac{1}{N}\sum_{k = 1}^N \Big\lvert\frac{1}{N}\sum_{\ell\ls N}\mathds{1}_{B(x_k, s)}\left( x_\ell \right) -2s - \frac{1}{N}\Big\rvert \\
& \ls D_N + \frac{1}{N},
	\end{align*}
	where
\[ D_N= D_N(\{x_1,\ldots,x_N\}) := \sup_{0 \ls a < b \ls 1} 
\Big\lvert\frac{1}{N}\#\left\{ i \ls N: x_i \in [a,b]\right\} - (b-a)\Big\rvert\] 
denotes the discrepancy of $\{x_1,\ldots,x_N\}$.
	Since $\xn$ is uniformly distributed,  $\lim\limits_{N \to \infty}D_N  = 0$ and hence it follows that $\xn$ has $0$\,--\,PPC.
\par Conversely, assume $(x_n)_{n \in \mathbb{N}}$ has $0$\,--\,PPC but is not uniformly distributed.
Then there exists an interval $[0,a] \subseteq [0,1]$ and a sequence $(N_k)_{k \in \mathbb{N}}$ such that
\begin{equation}\label{subsequence_0_PPC}\lim_{k \to \infty}  \frac{1}{N_k}\#\{n \ls N_k: x_n \in [0,a]\} = b \neq a.
\end{equation}
We now establish a relation between the correlation function $R_2(0;s,N_k)$ and the function $F_0(t,s,N_k)$ defined in \eqref{Fdef}. A counting argument gives 
\begin{align} 
   \frac{1}{N_k^2}& \#\Big\{1  \ls n \neq m \ls N_k:
\lVert x_n-x_m \rVert \ls s\Big\} \nonumber \\ 
& \quad \gs \frac{1}{N_k^2}\#\Big\{n\ls N_k: \|x_n-t\| \ls \frac{s}{2} \Big\}^2  - \frac{1}{N_k^2}\#\Big\{n\ls N_k: \|x_n-t\| \ls \frac{s}{2}\Big\}\label{counting} \\
 &\quad  \gs \Big(\frac{1}{N_k} \#\Big\{n\ls N_k: \|x_n-t\| \ls \frac{s}{2}\Big\}\Big)^2  - \frac{1}{N_k},\quad \text{ for any } 0\ls t \ls 1. \nonumber
\end{align}
By the assumption that $\xn$ has $0$\,--\,PPC, letting $k\to\infty$ in \eqref{counting} we get
\begin{align}
   2 s  &= \limsup_{k \to \infty} \frac{1}{N_k^2} \#\Big\{1 \ls n \neq m \ls N_k:
\lVert x_n-x_m \rVert \ls s \Big\} \nonumber
\\&\gs  \limsup_{k \to \infty} \Big(\frac{1}{N_k} \#\Big\{n\ls N_k: \|x_n-t\| \ls \frac{s}{2}\Big\}\Big)^2   \nonumber
\end{align}
for all $0\ls t \ls 1$ and all $0<s\ls \frac12$.  Thus if we fix some $\varepsilon>0,$ there exist some $s > 0$ and $K\in \mathbb{N}$ such that for all $k \gs K,$
\[\frac{1}{N_k}\#\Big\{n\ls N_k: \|x_n\| \ls \frac{s}{2}\Big\}\, < \, \frac{\varepsilon}{3} \]
and
\[\frac{1}{N_k}\#\Big\{n\ls N_k: \|x_n-a\| \ls \frac{s}{2}\Big\}\, < \, \frac{\varepsilon}{3} \cdot \]
In view of \eqref{subsequence_0_PPC}, we can additionally assume that for all $k\gs K$ we have
\[\frac{1}{N_k}\#\big\{n \ls N_k: x_n \in [0,a]\big\} > b -\frac{\varepsilon}{3} \cdot \]
Observe that the function $F_0(t,s,N)$ defined in \eqref{Fdef}
satisfies
\begin{eqnarray*}
\int_0^a F_0(t,s,N_k)\,\mathrm{d}t &=& \frac{1}{N_k}\sum_{n\ls N_k}\int_0^1 \one_{B\left(x_n,\frac{s}{2}\right)}(t)\one_{[0,a]}(t)\,\mathrm{d}t \\
&=& \frac{1}{N_k}\sum_{n\ls N_k}\lambda\Big(B(x_n, \tfrac{s}{2}) \cap [0,a] \Big) \\
&\gs & \frac{1}{N_k}\hspace{-2mm}\sum_{\substack{n\ls N_k\\ \frac{s}{2}\ls x_n \ls a-\frac{s}{2}}}\hspace{-3mm} \lambda\Big(B\big(x_n, \tfrac{s}{2}\big) \Big) \\
&=& \frac{s}{N_k} \#\Big\{n\ls N_k: \frac{s}{2}\ls x_n \ls a- \frac{s}{2} \Big\}.
\end{eqnarray*}
Thus for all $k \gs K$, 
\begin{align*}\frac{1}{s}\int_0^a F_0(t,s,N_k)\,\mathrm{d}t &\gs \frac{1}{N_k} \#\Big\{n\ls N_k: \frac{s}{2}\ls x_n \ls a- \frac{s}{2} \Big\} \\&\gs 
\frac{1}{N_k}\#\{n \ls N_k: x_n \in [0,a]\} \\& \hspace{4mm} - \frac{1}{N_k}\#\Big\{n\ls N_k: \|x_n\| \ls \frac{s}{2}\Big\}  - \frac{1}{N_k}\#\Big\{n\ls N_k: \|x_n-a\| \ls \frac{s}{2}\Big\} \\&\gs b- \ve  \end{align*}
and similarly,  we can show that 
\[\frac{1}{s}\int_a^1 F_0(t,s,N_k)\,\mathrm{d}t \gs 1 - b - \ve .  \]
By applying the Cauchy-Schwarz inequality, we obtain
\begin{eqnarray*}
 \int_0^1 F_0(t,s,N_k)^2 \,\mathrm{d}t &  = &  \int_0^a F_0(t,s,N_k)^2\,\mathrm{d}t   +    \int_a^1 F_0(t,s,N_k)^2 \,\mathrm{d}t  \\
	&\gs&  \frac{1}{a}\left( \int_0^a F_0(t,s,N_k)\,\mathrm{d}t\right)^2 + \frac{1}{1-a}\left( \int_a^1 F_0(t,s,N_k)\,\mathrm{d}t\right)^2 \\
    &\gs & \frac{s^2}{a}(b-\varepsilon)^2 + \frac{s^2}{1-a}(1-b-\varepsilon)^2.
\end{eqnarray*}
Letting $k\to \infty,$ the   left hand side converges to $s^2$ by the assumption of $0$\,--\,PPC (this follows by Lemma \ref{integral_equivalence}), therefore
\[1> \frac{(b-\varepsilon)^2}{a} + \frac{(1-b-\varepsilon)^2}{1-a} \cdot\]
This is a contradiction for $\varepsilon$ sufficiently small because $b\neq a.$
\begin{xrem} The arguments of the second part of the previous proof, i.e. the proof of the fact that $0$\,--\,PPC implies uniform distribution, can be used to deduce that $\beta$\,--\,PPC implies uniform distribution for any $\beta<1.$  When $\beta=1$, the same method can actually yield something more: any sequence for which \[\lim_{N\to\infty}R_2(s,N) = 2s \qquad \text{ for all } s\gs s_0 \] is uniformly distributed. Indeed, under this hypothesis a modification of the proof of Lemma \ref{integral_equivalence} gives that for all $s\gs s_0,$ 
\[\limsup_{N \to \infty} \int_0^1 F_1(t,s,N)^2\, \mathrm{d}t \ls s^2 + s + s_0^2.\]
The implicit constant here does not depend on the value of $s$. Assuming such a sequence is not uniformly distributed, we obtain \eqref{subsequence_0_PPC} and thus for all $k$ sufficiently large  we get 
\[\int_0^a F_1(t,s,N_k)\,\mathrm{d}t \gs (b-\varepsilon)s\quad  \text{ and }\quad   \int_a^1 F_1(t,s,N_k)\,\mathrm{d}t \gs (1-b-\varepsilon)s. \]
With an application of the Cauchy--Schwarz inequality this gives  
\begin{align}\label{cs_C2}
    \int_0^1 F_1(t,s,N_k)^2 \,\mathrm{d}t  \gs \frac{s^2}{a}(b-\varepsilon)^2 + \frac{s^2}{1-a}(1-b-\varepsilon)^2.
\end{align}
Letting $k\to \infty$ we obtain    
\[    \frac{s^2}{a}(b-\varepsilon)^2 + \frac{s^2}{1-a}(1-b-\varepsilon)^2 \ls s^2 + s + s_0^2 \quad \text{for all } s\gs s_0\]
 and if we choose $s>0$ sufficiently large, we arrive at a contradiction. We note that this proof can be seen as a ``continuous analogue'' of the proof of Aistleitner, Lachmann and Pausinger \cite{alp} that avoids the need to use the positivity of the Fej\'er kernel.  
\end{xrem}

\section{Proof of Theorem \ref{small_scales_suffice}}

 In this section we shall prove that whenever $0<\beta<1$ and 
 \begin{equation}\label{small_scale}
\lim_{N \to \infty}R_2(\beta;s,N) = 2s \qquad \text{ for all } s < s_0.\end{equation}
 holds for some value of $s_0>0,$ the sequence $\xn$ has $\beta$\,--\,PPC. In view of Lemmas \ref{integral_equivalence} and \ref{integral_lower_bound}, it suffices to show that for all $s> 0$ we have
\begin{equation}\label{larger_scale}\limsup_{N \to \infty} I_{\beta}(s,N) \ls s^2.\end{equation} 
To do so, note that for any fixed $s>0$, there exists an $N_0= N_0(s,s_0)$ such that for any even integer $K \gs N_0$ we have
$s/K < s_0$.
Note that when $N$ is sufficiently large, for any $t\in [0,1]$ we have
\[ B\Big(t, \frac{s}{2N^{\beta}}\Big) \subseteq \bigcup_{|\ell| \ls \frac{K}{2} }\hspace{-1mm} B\Big(t + \frac{\ell s }{KN^\beta}, \frac{s}{2KN^\beta}\Big)  \]
and therefore,
\begin{equation*} 
F_\beta(t,s,N) \ls \sum_{|\ell|\ls  \frac{K}{2}} F_\beta\Big(t + \frac{ \ell s}{KN^{\beta}},\frac{s}{K},N\Big).\end{equation*}
The Cauchy\,--\,Schwarz inequality gives 
\begin{align*}
\int_{0}^{1} F_{\beta}(t,s,N)^2 \,\mathrm{d}t 
&\ls \int_{0}^{1} \Big(\sum_{ |\ell|\ls   \frac{K}{2} } F_\beta\Big(t + \frac{ \ell s }{KN^{\beta}},\frac{s}{K},N\Big)\Big)^2 \,\mathrm{d}t 
\\& \ls \int_{0}^{1} (K +1) \sum_{|\ell|\ls  \frac{K}{2} }F_\beta\Big(t + \frac{ \ell s}{KN^{\beta}},\frac{s}{K},N\Big)^2 \, \mathrm{d}t
\\&= (K +1)^2   \int_{0}^{1} F_{\beta}\Big(t,\frac{s}{K},N\Big)^2 \, \mathrm{d}t.
\end{align*}
By the assumption of \eqref{small_scale} and using Lemma \ref{integral_equivalence},
\[\lim_{N \to \infty}\int_{0}^{1} F_\beta\Big(t,\frac{s}{K},N\Big)^2 \,\mathrm{d}t = \frac{s^2}{K^2},\]
so we obtain 
\[
\limsup_{N \to \infty} \int_{0}^{1} F_{\beta}(t,s,N)^2 \,\mathrm{d}t \ls \frac{(K+1)^2}{K^2}s^2.
\]
Letting $K \to \infty$, the result follows.
 
\section{Proof of Theorem \ref{stronger_not_for_one}}
We prove that given any $S>0,$ we can find a sequence $\xn\subseteq [0,1]$ such that on the one hand $R_2(s,N)\to 2s$ for all $s<S$ but on the other hand, $\xn$ is not uniformly distributed in $[0,1].$  During the course of the proof we shall deal with the correlation functions of different sequences at the same time, and thus for convenience we stress out the dependence of the correlation function on the sequence.
\par We begin with a brief description of the main heuristic idea of the proof. Given the value $S>0$, we first want to define a sequence $\xn$ whose pair correlation function is asymptotically smaller than the value that corresponds to the Poisson model for all $s<S.$ More specifically, we want to find a sequence such that  $R_2(s,N) \to 2cs$ for all $s<\frac{1}{c}S$, where $0<c<1$ is some constant. We then expect that if we contract the sequence by a factor $c,$ on the one hand the ``contracted'' sequence will have pair correlations asymptotically equal to $2s$ for all $s<S,$ on the other hand it cannot be uniformly distributed modulo $1,$ since it will be contained within the interval $[0,c].$ 
\par We thus need to take a careful look on how contracting a sequence affects its pair correlation function. Let $\xn$ be a sequence in $[0,1]$ and $0 < c < 1$ be an arbitrary constant. For any two distinct terms $x_n,x_m \in [0,1]$, 
we have \begin{equation} \label{rescaling_equiv} \begin{aligned}
    \|c(x_n - x_m) \| \ls \frac{s}{N}\quad &\Leftrightarrow \quad  \exists k \in \mathbb{Z}: \lvert c(x_n - x_m) - k\rvert \ls \frac{s}{N}
    \\& \Leftrightarrow \quad 
    \exists k \in \mathbb{Z}: \big\lvert x_n - x_m - k/c\big\rvert \ls \frac{s/c}{N} \cdot
\end{aligned} \end{equation}
We now observe that for all $N\gs 1$ sufficiently large, the only candidate for $k$ is $0$: If $k \neq 0$, then $\lvert k/c \rvert \gs 1 + \delta$ for some fixed $\delta > 0$. Now if $N\gs 1$ is  large enough such that $\dfrac{s/c}{N} < \delta$, since
$x_n - x_m \in [-1,1]$ we see that $\lvert x_n - x_m - \frac{k}{c}\rvert  > \dfrac{s/c}{N}$ and the inequalities in \eqref{rescaling_equiv} will have to fail. Hence,
\begin{align}\|c(x_n - x_m) \| \ls \frac{s}{N} \quad
&\Leftrightarrow \quad \lvert x_n - x_m\rvert \ls \frac{s/c}{N}
\\&\Leftrightarrow \quad \| x_n - x_m \| \ls \frac{s/c}{N}\hspace{3mm} \text{ and }\hspace{3mm} \lvert x_n - x_m \rvert < 1 - \frac{s/c}{N} \cdot
\end{align}
Therefore writing $R_2^{\mathfrak{X}}(s,N)$ and $R_2^{c\mathfrak{X}}(s,N)$ for the correlation functions of the sequences $\xn$ and $(cx_n)_{n\in\bN},$ respectively, we have

\begin{align}
R_2^{c\mathfrak{X}}(s,N) &= \frac{1}{N}\#\Big\{n \neq m \ls N: \| x_n - x_m \| \ls \frac{s/c}{N}\, \,  \& \, \, \lvert x_n - x_m \rvert < 1 - \frac{s/c}{N}\Big\}
\\&= \frac{1}{N}\#\Big\{ n\neq m \ls N: \| x_n - x_m \| \ls \frac{s/c}{N}\Big\} \\&\quad - \frac{1}{N}\#\Big\{ n \neq m \ls N: \| x_n - x_m \| \ls \frac{s/c}{N} \, \, \& \, \, \lvert x_n - x_m \rvert \gs 1 - \frac{s/c}{N} \Big\}
\\&= R_2^{\mathfrak{X}}(s/c,N)  -  E^{\fX}(s/c,N), 
\end{align}
where $E^{\fX}(s,N)$ is an error term, by definition equal to 
\begin{equation} \label{error}
 E^{\fX}(s,N) = \frac{1}{N}\#\Big\{ n \neq m \ls N: \| x_n - x_m \| \ls \frac{s}{N}\,\&\,\lvert x_n - x_m \rvert \gs 1 - \frac{s}{N}\Big\}.
\end{equation}
The upshot is that for sequences $\xn$ for which the error term $E^{\fX}(s/c,N)$ tends to $0$ as $N\to \infty,$ we will have
\begin{equation}\label{scaling_seq_lim}
\lim_{N \to \infty} R_2^{c\mathfrak{X}}(s,N) = \lim_{N \to \infty} R_2^{\mathfrak{X}}(s/c,N),\end{equation}
provided, of course, that the limit on the right-hand side exists. So as long as  we find a sequence $\xn$ with $E^{\fX}(s,N)\to 0$ for any $s>0$ and a number $0<c<1$ such that the correlation function $R_2^{\mathfrak{X}}(s,N)$ of $\xn$ satisfies
\begin{equation}\label{suffices_thm3}\lim_{N \to \infty} R_2^{\mathfrak{X}}(s,N) = 2cs  \qquad \text{for any } s < \frac{1}{c}S,\end{equation}
the proof of Theorem \ref{stronger_not_for_one} can be finished as already mentioned: the sequence $(cx_n)_{n\in\bN}$ satisfies  $\lim_{N \to \infty} R_2^{c\mathfrak{X} }(s,N) = 2s$ for every $s<  S$, but is only supported in $[0,c]$, which makes it impossible to be uniformly distributed.
\par We now move our discussion to how we can construct a sequence that satisfies  \eqref{suffices_thm3}. The idea is to start with some sequence $\yn$ that has PPC and ``dilute'' its pair correlations by inserting periodically terms from some other sequence $\zn$, for which the pair correlation function is $0$ for all sufficiently small scales $s>0.$  \par More precisely, we  consider a positive integer $M \gs 2S$, a sequence $\yn$ that will be specified later, as well as the binary van der Corput sequence $\zn$. By the definition of $\zn$ (see e.g. \cite[p. 127]{kuipers}) one can deduce that for any \mbox{$N \gs 1,$} all gaps between elements of the set $\{z_1,\ldots ,z_N\}$ are at least $1/(2N).$ It  follows immediately that the correlation function $R_2^{\mathcal{Z}}(s,N)$ of $\zn$ satisfies
\begin{equation}\label{vdc_no_ppc}
\lim_{N \to \infty}R_2^{\mathcal{Z}}(s,N) =  0 \qquad \text{ for any }\, 0 < s < \frac12 \cdot \end{equation}
We now build $\xn$ by inserting an element of $\zn$ after every $M-1$ elements of $\yn$; that is, writing $n = kM + r$ where $k \in \mathbb{N}$ and  $0 \ls r < M$, we define
\begin{equation} \label{xnt_def}
x_n = \begin{cases} z_k, &\text{ if $r = 0$} \\ 
y_{k(M-1)+r}, & \text{ if $r \neq 0$}. \end{cases}  \end{equation}  We then compute
\begin{align}
R_2^{ \mathfrak{X}}(s,MN) &= \frac{1}{MN}\#\Big\{n \neq m \ls (M-1)N: \|y_n - y_m\| \ls \frac{s/M}{N}\Big\}
    \\&\quad + \frac{1}{MN}\#\Big\{n \neq m \ls N: \|z_n - z_m\| \ls \frac{s/M}{N}\Big\}
    \\& \quad + \frac{2}{MN}\#\Big\{n \ls (M-1)N, m \ls N: \|y_n - z_m\| \ls \frac{s/M}{N}\Big\}\\
    &= \frac{M-1}{M}R_2^{\Upsilon}\left(s(M-1)/M,(M-1)N\right) \label{r2_of_new_seq} + \frac{1}{M}R_2^{\mathcal{Z}}(s/M,N)
    \\&\quad + \frac{2}{MN}\#\Big\{n \ls (M-1)N, m \ls N: \|y_n - z_m\| \ls \frac{s/M}{N}\Big\}.
\end{align}
In view of \eqref{vdc_no_ppc} and the assumption that $M \gs 2S$, we see that
\mbox{$\lim_{N \to \infty} R_2^{\mathcal{Z}}(s/M,N) = 0$} for any $s < S$. So we are close to the end of the proof as long as we find a sequence $\yn$ fulfilling the following properties: \newline 
${\rm (i)}$ $\yn$ has PPC, \\
${\rm (ii)}$  for every $s > 0,$  
    \begin{equation}\label{correl_with_vdc}
 \lim_{N \to \infty} \frac{1}{N}\#\Big\{n \ls (M-1)N, m \ls N: \|y_n - z_m\| \ls \frac{s/M}{N}\Big\} = \frac{2s(M-1)}{M},  
\end{equation}
${\rm (iii)}$ for every $s>0,$ the error term $E^{\Upsilon}(s,N)$ defined in \eqref{error} corresponding to the sequence $\yn$ satisfies $\lim\limits_{N\to\infty}E^{\Upsilon}(s,N)=0$.\newline 

\noindent The existence of such a sequence $\yn$ is guaranteed by the following lemma.
\begin{lem}\label{almost_sure_lem}
Let $M\gs 1$ be an integer and let $\zn$ denote the binary van der Corput sequence. Furthermore let $(Y_n)_{n\in\mathbb{N}}$ be a sequence of independent, uniformly distributed random variables in $[0,1]$. Then almost surely, the sequence $(Y_n(\omega))_{n\in\mathbb{N}}$ has PPC,  satisfies property \eqref{correl_with_vdc} and $\lim\limits_{N\to\infty}E(s,N)=0$.
\end{lem}
\begin{proof}
It is a well\,--\,known fact that any sequence $(Y_n)_{n\in\mathbb{N}}$ of independent, uniformly distributed random variables in $[0,1]$ has PPC  almost surely. For a proof, we refer to \cite[Appendix B]{our_other_paper} in the setup of higher order correlations, or also to \cite{hinrichs} for sequences in higher dimensions. \par
We now prove the second property, namely that $(Y_n)_{n\in\mathbb{N}}$ satisfies \eqref{correl_with_vdc} almost surely.  Choose an  arbitrary $s > 0$. Writing 
\[Y_{n,N}(s) = \#\Big\{1\ls m \ls N: \|Y_n - z_m\| \ls \frac{s}{MN}\Big\}, \quad  n = 1, \ldots, (M-1)N,\]
the quantity in question is equal to 
\[ I_{N}(s) =  \frac{1}{N} \sum_{n \ls (M-1)N} Y_{n,N}(s).\]
It is easy to see that 
\[\mathbb{E}[ I_N(s) ] = \frac{2s(M-1)}{M},  \qquad N\gs 1 \]
and we proceed to compute the variance of $I_N.$ We first observe that for any $N\gs 1$, the functions $(Y_{n,N})_{n=1}^{N(M-1)}$ form a family of independent random variables.  Furthermore, for $2^k \ls N < 2^{k+1}$  we have
$(z_m)_{m=1}^{N} \subseteq \{\frac{\ell}{2^{k+1}}, \ell = 1,\ldots, 2^{k+1}\}$,
so we can bound
\[Y_{n,N}(s) \ls \#\Big\{1 \ls m \ls 2^{k+1}: \frac{m}{2^{k+1}} \in \Big[y_n - \frac{2s/M}{2^{k+1}}, 
y_n + \frac{2s/M}{2^{k+1}}\Big]\Big\}
\ls \frac{4s}{M} + 2.\]
This implies that
\begin{align} \label{variance}
    \text{Var}[I_N]\! =\!\frac{1}{N^2}\hspace{-3mm}\sum_{n=1}^{N(M-1)}\hspace{-3mm}\text{Var}[Y_{n,N}] = \frac{1}{N^2}\hspace{-3mm}\sum_{n=1}^{N(M-1)}\hspace{-3mm}\Big( \mathbb{E}[Y_{n,N}^2] 
    -\mathbb{E}[Y_{n,N}]^2\Big) \!=\! \mathcal{O}_{M,s}\Big(\frac{1}{N}\Big).
\end{align}
To finish the proof we use a standard approximation argument (for more details see e.g. \cite{rz2}). Let  $\gamma > 0$ and consider the subsequence $$B_N := \left\lceil{N^{\gamma}}\right\rceil,\qquad  N\gs 1.$$
For $s>0$ fixed, we use Chebyshev's inequality, the first Borel--Cantelli Lemma and the variance estimate from \eqref{variance} to see that
 \begin{equation}\label{correl_vdc_sub}\lim_{N \to \infty} I_{B_N}(s) = \frac{2s(M-1)}{M} \qquad \text{ almost surely }
 \end{equation}
  (where the zero\,--\,measure set depends on $s$). Repeating the argument for all $s$ lying in a dense, countable subset of $\mathbb{R}_+$ and employing the monotonicity of $I_N(s)$ as a function of $s$, we see that \eqref{correl_vdc_sub} actually holds for all $s>0$. Next, if   $N\gs 1$ is an arbitrary integer, we let $K\gs 1$ be such that $B_K \ls N < B_{K+1}$ and observe that for any $s>0,$
\begin{align} \label{subsequence_suffices}
\frac{B_K}{B_{K+1}} I_{B_K}\Big(\frac{B_Ks}{B_{K+1}}\Big) \ls I_N(s) \ls \frac{B_{K+1}}{B_K}I_{B_{K+1}}\Big( \frac{B_{K+1}s}{B_K}\Big).
\end{align}
Since $\lim\limits_{N \to \infty} \dfrac{B_N}{B_{N+1}} = 1$, we deduce from \eqref{correl_vdc_sub} that \eqref{correl_with_vdc} holds almost surely for any fixed $s > 0$.   
\par Finally, it remains to prove that the third property, namely $\lim_{N\to \infty} E(s,N)=0,$ is satisfied almost surely. For the error term $E^{\Upsilon}(s,N)$ corresponding to $\yn$, we  provide the upper bound
\begin{equation} \label{error_upper_bound}  \begin{aligned} 
 E^{\Upsilon}(s,N) &\ls \frac{1}{N}\#\Big\{ n \neq m \ls N: \lvert y_n - y_m \rvert \gs 1 - \frac{s}{N}\Big\}\\
 &\ls \frac{2}{N} \#\Big\{ n \ls N: y_n \gs 1 - \frac{s}{N}\Big\}\cdot\#\Big\{ n \ls N: y_n \ls  \frac{s}{N} \Big\}.
\end{aligned} \end{equation} 
It is therefore sufficient to prove that almost surely, we have
\begin{equation} \label{error_to_zero}
    \lim_{N \to \infty} \frac{1}{N}\#\Big\{ n \ls N: Y_n \gs 1 - \frac{s}{N}\Big\}\cdot\#\Big\{ n \ls N: Y_n \ls  \frac{s}{N} \Big\} =0.
\end{equation} 
This can be proved using a mean\,--\,variance argument, completely analogous to the one used to prove \eqref{correl_with_vdc}. We leave the details to the interested reader.
\end{proof}
Having proved the existence of a sequence $\yn$ with the three desired properties as in Lemma \ref{almost_sure_lem}, we define $\xn$ as in \eqref{xnt_def}. In view of these properties,  \eqref{r2_of_new_seq} implies that for any $s>0,$ the pair correlation $R_2^{\fX}(s,N)$ of $\xn$ satisfies
\begin{equation}
   R_2^{\fX}(s, MN) = \Big(1 - \frac{1}{M^2}\Big)\,2s +  \frac{1}{M} R_2^{\mathcal{Z}}(s/M,N) + o(1), \quad N\to \infty. 
\end{equation}
From now on, we focus our attention on values $s<M/2.$ By the assumption that $M\gs 2S,$ these values include all scales $s<S.$ Since for $s < M/2$ we have $ R_2^{\mathcal{Z}}(\frac{s}{M}, MN\big) \to 0$, we deduce that for $s< M/2,$  
\[ R_2^{\fX}(s, MN) = \Big(1 - \frac{1}{M^2}\Big)\,2s + o(1),\qquad N\to\infty.\] Then a standard approximation argument, similar to the one we used to derive \eqref{subsequence_suffices}, gives
\begin{equation} \label{penultimate}
\lim_{N \to \infty}R_2^{\fX}(s, N) =  \Big(1 - \frac{1}{M^2}\Big)\,2s \qquad \text{ for all } s< \frac{M}{2}\cdot 
\end{equation}
Finally, we set $c = 1 - \dfrac{1}{M^2}\in (0,1)$ and consider the sequence $(cx_n)_{n\in\bN}.$ In order to prove that this sequence satisfies the statement of Theorem \ref{small_scales_suffice}, it remains to verify that the error term $E^{\fX}(s,N)$ corresponding to $(x_n)_{n\in\bN}$ tends to $0$ for any $s>0.$ As in \eqref{error_upper_bound}, we find \vspace{-4mm}
\begin{align}\label{alternative_error_estimate}
E^{\fX}(s,MN) & = \frac{1}{MN}\#\Big\{ n \neq m \ls MN: \| x_n - x_m \| \ls \frac{s}{MN}\,\&\,\lvert x_n - x_m \rvert \gs 1 - \frac{s}{MN}\Big\}\\
     &\ls \frac{M-1}{M}E^{\Upsilon}(s(M-1)/M,(M-1)N) + \frac{1}{M}E^{\mathcal{Z}}(s/M,N)
     \\&\,\,+ \frac{1}{MN}\#\Big\{n \ls (M-1)N: y_n \ls \frac{s/M}{N}\Big\}\cdot \#\Big\{n \ls N: z_n \gs 1-\frac{s/M}{N}\Big\}
    \\&\,\,+ \frac{1}{MN}\#\Big\{n \ls (M-1)N: y_n \gs 1-\frac{s/M}{N}\Big\}\cdot \#\Big\{n \ls N: z_n \ls \frac{s/M}{N}\Big\}.
\end{align}
Using that $s/M < 1/2$ and the property 
$\{z_1,\ldots z_N\} \subseteq [1/(2N),1-1/(2N)]$ that follows straightforwardly from the definition of the binary van der Corput sequence, we see that
\[\#\Big\{n \ls N: z_n \gs 1-\frac{s/M}{N}\Big\} = \#\Big\{n \ls N: z_n \ls \frac{s/M}{N}\Big\} = 0,\]
whence $E^{\mathcal{Z}}(s/M,N)=0$ and  
\[E^{\fX}(s,MN) \ls \frac{M-1}{M}E^{\Upsilon}(s(M-1)/M,(M-1)N).\]
By the construction of $\yn$, $E^{\Upsilon}(s(M-1)/M,(M-1)N) \to 0$ as $N\to\infty$, so
$\lim\limits_{N\to\infty}E^{\fX}(s,MN)=0$ and subsequently
$\lim\limits_{N\to\infty}E^{\fX}(s,N)=0$ follows.
According to the discussion in the beginning of the section, this allows us to deduce that for the rescaled sequence $(cx_n)_{n\in\bN}$ we have
\[ R_2^{c\mathfrak{X}}(s,N) = R_2^{\mathfrak{X}}(s/c,N) + o(1), \quad N\to \infty. \]
Combined with  \eqref{penultimate}, this implies that 
\[ \lim_{N\to\infty} R_2^{c\fX}(s,N) = 2s \quad \text{ for all } s<S\]
which concludes the proof of Theorem \ref{stronger_not_for_one} by the discussion above.

\section{ Proof of Theorem  \ref{beta_implies_alpha} }  
We  now proceed to the proof of Theorem \ref{beta_implies_alpha}: we prove that when $\alpha<\beta$, the property of $\beta$\,--\,PPC is stronger than that of $\alpha$\,--\,PPC.  We shall distinguish two different cases according to whether $\beta=1$ or $\beta<1.$ In both cases, by Lemma \ref{integral_equivalence} and Lemma \ref{integral_lower_bound}, it suffices to show that for all $s > 0$
\[\limsup_{N \to \infty} I_\alpha(s,N) \ls s^2.\]

\subsection{The case $0<\beta <1$}  Let $0\ls \alpha < \beta < 1.$  We bound the function $F_\alpha(t,s,N)$ from above in terms of $F_\beta(t,s,N)$. Writing
$$ M = \Big\lceil \frac{ N^{\beta-\alpha}}{2} \Big\rceil,$$ 
we use the same reasoning as in the proof of Theorem  \ref{small_scales_suffice}: we note that
\[ B\Big(t, \frac{s}{2N^{\alpha}}\Big) \subseteq \bigcup_{|\ell| \ls M }\hspace{-1mm} B\Big(t + \frac{\ell s}{N^\beta}, \frac{s}{2N^\beta}\Big),\]
whence
\begin{equation}\label{beta_alpha_eq}
F_\alpha(t,s,N) \ls \frac{1}{ N^{\beta - \alpha}}\sum_{|\ell|\ls  M} F_\beta\Big(t + \frac{ \ell s}{N^{\beta}},s,N\Big).
\end{equation}
Another application of the Cauchy-Schwarz inequality gives 
\begin{align*}
\int_{0}^{1} F_{\alpha}(t,s,N)^2 \,\mathrm{d}t 
&\ls \int_{0}^{1} \frac{1}{N^{2(\beta - \alpha)}}\Bigg(\sum_{ |\ell|\ls   M } F_\beta\Big(t + \frac{ \ell s}{N^{\beta}},s,N\Big)\Bigg)^2 \,\mathrm{d}t 
\\& \ls \int_{0}^{1} \frac{ N^{\beta - \alpha}+1 }{N^{2(\beta - \alpha)}}  \sum_{|\ell|\ls  M }F_\beta\Big(t + \frac{ \ell s}{N^{\beta}},s,N\Big)^2 \,\mathrm{d}t
\\&= \frac{N^{\beta - \alpha}+1 }{N^{2(\beta - \alpha)}} \sum_{|\ell|\ls  M} \int_{0}^{1} F_{\beta}(t,s,N)^2 \,\mathrm{d}t \\& \ls
\Big(1 + \frac{3 }{ N^{\beta-\alpha}} \Big) \int_{0}^{1} F_{\beta}(t,s,N)^2 \,\mathrm{d}t.
\end{align*}
By the assumption of $\beta$\,--\,PPC and Lemma \ref{integral_equivalence}, we have
\[\lim_{N \to \infty}\int_{0}^{1} F_\beta(t,s,N)^2 \mathrm{d}t = s^2,\]
so the result follows.

\subsection{The case $\beta=1$} 
Assume the sequence $\xn$ has PPC and let $s > 0$. For all integers $K, N \gs 1$ we set 
\[M = M(N,K) = \Bigg\lceil \frac{N^{1-\alpha}}{2K}\Bigg\rceil. \]
As in the previous case, we have 
\[ B\Big(t, \frac{s}{2N^{\alpha}}\Big) \subseteq \bigcup_{|\ell| \ls M }\hspace{-1mm} B\Big(t + \frac{\ell sK}{N}, \frac{sK}{2N}\Big)  \]
and therefore
\begin{equation*}
F_\alpha(t,s,N) \ls \frac{1}{ N^{1- \alpha}}\sum_{|\ell|\ls  M} F_1\Big(t + \frac{ \ell sK}{N},sK,N\Big).
\end{equation*}
By the Cauchy-Schwarz inequality and argumentations as previously, it follows that
\begin{align*}
\int_{0}^{1} F_{\alpha}(t,s,N)^2 \,\mathrm{d}t 
&\ls \int_{0}^{1} \frac{1}{N^{2(1-\alpha)}}\Big(\sum_{ |\ell|\ls   M } F_1\Big(t + \frac{ \ell sK}{N},sK,N\Big)\Big)^2 \,\mathrm{d}t 
\\& \ls \int_{0}^{1} \frac{2M+1}{N^{2(1-\alpha)}}  \sum_{|\ell|\ls  M }F_1\Big(t + \frac{ \ell sK}{N},sK,N\Big)^2 \,\mathrm{d}t
\\&= \frac{(2M+1)^2}{N^{2(1-\alpha)}} \int_{0}^{1} F_{1}(t,sK,N)^2 \,\mathrm{d}t.
\end{align*}
Clearly,
$$\lim_{N\to\infty} \frac{(2M+1)^2}{N^{2(1-\alpha)}} = \frac{1}{K^2}$$
and by the assumption of PPC we have
\[\lim_{N \to \infty}\int_{0}^{1} F_1(t,sK,N)^2\, \mathrm{d}t = (sK)^2 + sK.\]
Therefore, we can deduce that
\[\limsup_{N \to \infty} \int_{0}^{1} F_{\alpha}(t,s,N)^2 \,\mathrm{d}t 
\ls \frac{s^2K^2 + sK}{K^2}\]
for any $K \in \mathbb{N}$. Letting $K \to \infty$, the statement follows.
\section{Proof of Theorem \ref{beta_threshold}}

In the current section we fix a value of $0< \beta <1$ and construct a sequence $\xn$ that has $\alpha$\,--\,PPC for all $\alpha<\beta$ but does not have $\alpha$\,--\,PPC for any $\alpha \gs \beta.$ The construction is based on the following idea: we start with an arbitrary sequence $(y_n)_{n\in\bN}$ that has PPC. The required sequence $\xn$ is then defined in a way that for each $1 \ls n \ls N$,  the number $y_n$   occurs asymptotically $N^{\tfrac{1}{\beta}-1}$ times in the first $N^{\tfrac{1}{\beta}}$ elements of $\xn$. In that way, the contribution to $R_2(\alpha; s,N)$ of terms $x_i,x_j$ with $1\ls  i\neq j \ls N$ and $x_i = x_j$ will be of order $N^{2-\beta}$.
Looking at the  factor $\frac{1}{N^{2-\alpha}}$ in the definition of $R_2(\alpha;s,N)$, we see
that the contribution of these pairs is negligible when $\alpha < \beta$, whereas for $\alpha > \beta$, it is impossible to have the Poissonian correlation property since the correlation function will diverge.\par To achieve the aforementioned construction, we employ an ``expansion'' of positive integers with respect to a sequence that grows like $\frac{1}{\beta}$\,--\,th powers of integers. \par 
\vspace{4mm}

Fix $0 < \beta <1$ and let $(y_n)_{n\in\bN}$ be a sequence with PPC such that $y_m \neq y_n$ whenever $m\neq n$. (The existence of such a sequence is justified, for example, by \mbox{\cite[Theorem 1]{rs}}.)
Also, consider the sequence of indices 
\[ A_N = N \lfloor N^{\frac{1}{\beta}-1 } \rfloor, \qquad N\gs 1.  \]
Every $N\gs 1$ can be written uniquely as 
\begin{equation} \label{beta_expansion}
N = A_M + \varepsilon_N\lfloor M^{\frac{1}{\beta}-1}\rfloor + q_N(M+1) + r_N, 
\end{equation}
where
\begin{enumerate}
    \item[(i)] $M = M_N \gs 1$ is the unique integer such that $A_M < N \ls A_{M+1},$
    \item[(ii)] $\varepsilon_N \in \{ 0, 1\}, \quad q_N \gs 0, \quad 1 \ls r_N \ls M+1, \quad $ and
    \item[(iii)] $\varepsilon_N = 0$ \text{ if and only if } $A_M < N \ls A_M + \lfloor M^{\frac{1}{\beta} - 1}\rfloor.$
\end{enumerate}
\vspace{3mm}

\noindent We now define a new sequence $\xn$ by letting $x_1 = y_1$ and 
\[ x_N = \begin{cases} y_{M+1}, &\text{ if } \varepsilon_N =0, \\ 
y_{r_N }, &\text{ otherwise.}
\end{cases}\]
Here $M\gs 1, \varepsilon_N \in \{0,1\} $ and $r_N\gs 1$ are as in \eqref{beta_expansion}. \par 

\vspace{2mm}

We aim to show that $\xn$ has the required property, that is, it does not have $\beta$\,--\,PPC but it has $\alpha$\,--\,PPC for any $\alpha<\beta.$ We first need the following lemma.

\begin{lem}\label{same_copies} For any $M\gs 1,$ we have
\begin{equation}\label{number_of_copies} \#\{1\ls n \ls A_M : x_n = y_k\} = \lfloor M^{\frac{1}{\beta} -1}\rfloor, \quad k=1,2,\ldots, M. \end{equation}
\end{lem}
\begin{proof}
We prove this statement by induction on $M$. Trivially, the statement holds for $M=1$, so we assume \eqref{number_of_copies} is true for some fixed $M \gs 1$. By construction of the sequence, $x_N = y_{m+1}$ for $A_M < N \ls A_M + \lfloor M^{\frac{1}{\beta} - 1}\rfloor = (M+1)\lfloor M^{\frac{1}{\beta} - 1}\rfloor.$
This implies that for all $\ell,k \in [M+1]$,
\begin{equation}\label{all_same_copies}\#\{ n \ls (M+1)\lfloor M^{\frac{1}{\beta} - 1}\rfloor : x_n = y_k\} = \#\{ n \ls (M+1)\lfloor M^{\frac{1}{\beta} - 1}\rfloor : x_n = y_{\ell}\}.\end{equation}
Since both $A_{M+1}$ and $A_M +\lfloor M^{\frac{1}{\beta} - 1}\rfloor =  (M+1)\lfloor M^{\frac{1}{\beta} - 1}\rfloor$ are divisible by $M+1$, so is their difference. Thus, as $N$ runs through the values $(M+1)\lfloor M^{\frac{1}{\beta} - 1}\rfloor< N \ls A_{M+1}$, it attains each residue class mod ${M+1}$  the same number of times. This means that $r_N$ attains each value between $1$ and $M+1$ the same number of times, and finally $x_N$ assumes the values $y_1,\ldots, y_{M+1}$ equally often. Hence, we can extend \eqref{all_same_copies} to $1 \ls n \ls A_{M+1}$.
Note that no element $y_k$ with $k \gs M+2$ appears in $\{x_1,\ldots, x_{A_{M+1}}\}$, so the statement follows.
\end{proof}
We can now show that $\xn$ does not have $\beta$\,--\,PPC. Let $s>0$. Whenever $m,n \ls A_M,$ the inequality $\|x_m - x_n\| \ls s/A_M^{\, \beta}$ is satisfied: 
\begin{enumerate}
\item[$\bullet$] either when $x_m=y_{k}$ and $x_n=y_\ell$ for some indices $1\ls k,\ell \ls M$ with $k\neq \ell$ such that $\|y_k - y_\ell\| \ls s/A_M^{\, \beta},$ or 
\item[$\bullet$] when $x_n=x_m=y_k$ for some $1\ls k\ls M.$ \end{enumerate}
Therefore by Lemma \ref{same_copies}, as $M\to \infty$ we have 
\begin{align*}
R_2(\beta; s, A_M) &= \frac{1}{A_M^{2-\beta}}\#\Big\{ 1\ls m \neq n \ls A_M : \| x_m - x_n \| \ls \frac{s}{A_M^{\, \beta}} \Big\} \\
& =  \frac{M\lfloor M^{\frac{1}{\beta}-1}\rfloor^2}{A_M^{2-\beta}} \cdot \frac{1}{M}\#\Big\{1\ls k \neq \ell \ls M : \|y_k - y_\ell\| \ls \frac{s}{A_M^{\beta}} \Big\} \\
 & \qquad+ \frac{1}{A_M^{2-\beta}}   \lfloor M^{\frac{1}{\beta}-1} \rfloor \big(\lfloor M^{\frac{1}{\beta}-1} \rfloor -1 \big)  M \\
&=  \big(1 + o(1) \big)\cdot \frac{1}{M}\#\Big\{1\ls k \neq \ell \ls M : \|y_k - y_\ell\| \ls \frac{s}{A_M^{\beta}} \Big\} + 1 +o(1).
\end{align*}
Let $\varepsilon>0.$ Since $\lim\limits_{M\to \infty} \dfrac{A_M^{\beta}}{M} =1, $
for all $M\gs 1$  large enough we have 
\begin{equation} \label{comparision}
\begin{aligned}
R_2^{\Upsilon}((1-\varepsilon)s,M) &= \frac{1}{M}\#\Big\{1\ls k \neq \ell \ls M : \|y_k - y_\ell\| \ls (1-\varepsilon)\frac{s}{M} \Big\} \\
&\ls \frac{1}{M}\#\Big\{1\ls k \neq \ell \ls M : \|y_k - y_\ell\| \ls \frac{s}{A_M^{\beta}} \Big\}  \\
& \ls  \frac{1}{M}\#\Big\{1\ls k \neq \ell \ls M : \|y_k - y_\ell\| \ls (1+\varepsilon)\frac{s}{M} \Big\} \\  
& = R_2^{\Upsilon}((1+\varepsilon)s,M),\end{aligned} \end{equation}
where we write $R_2^{\Upsilon}(s,N)$ for the pair correlation function of the sequence $\yn$ By the hypothesis that $\yn$ has PPC, we deduce that 
\[\lim_{M\to\infty} \frac{1}{M}\#\Big\{1\ls k \neq \ell \ls M : \|y_k - y_\ell\| \ls \frac{s}{A_M^{\beta}} \Big\} = 2s.   \]
Therefore,
\[ \lim_{M\to\infty} R_2(\beta; s, A_M) =  2s + 1, \]
which implies that $\xn$ does not have $\beta$\,--\,PPC. By Theorem \ref{beta_implies_alpha},  it does not have $\alpha$\,--\,PPC for any $\alpha>\beta.$ (We leave it to the interested reader to verify that actually for $\alpha>\beta,$  $\lim\limits_{N\to\infty}R_2(\alpha;s,N)=\infty$ for any $s>0.)$  
\vspace{2mm}

\noindent We now fix $0 \ls \alpha <\beta.$ Arguing as previously we obtain
\begin{align*}
R_2(\alpha;s,A_M) &= \frac{1}{A_M^{2-\alpha}}  \#\Big\{ 1\ls m \neq n \ls A_M : \| x_m - x_n \| \ls \frac{s}{A_M^\alpha} \Big\} \\
& =  \frac{ \lfloor M^{\frac{1}{\beta}-1}\rfloor^2}{A_M^{2-\alpha}} \cdot  \#\Big\{1\ls k \neq \ell \ls M : \|y_k - y_\ell\| \ls \frac{s}{A_M^{\alpha}} \Big\} \\
 & \qquad+ \frac{1}{A_M^{2-\alpha}}   \lfloor M^{\frac{1}{\beta}-1} \rfloor \big(\lfloor M^{\frac{1}{\beta}-1} \rfloor -1 \big)  M.
\end{align*}
We first notice that since $\alpha < \beta,$
\[\frac{M}{A_M^{2-\alpha}}   \lfloor M^{\frac{1}{\beta}-1} \rfloor \big(\lfloor M^{\frac{1}{\beta}-1} \rfloor -1 \big)   = 
\frac{1}{M^{1 - \frac{\alpha}{\beta}}} + o(1) = o(1),\qquad M\to \infty.  \]
It remains to understand the asymptotic behaviour of the factor
\[   \#\Big\{1\ls k \neq \ell \ls M : \|y_k - y_\ell\| \ls \frac{s}{A_M^{\alpha}} \Big\}.\]
By the definition of  $(A_M)_{M \in \mathbb{N}},$ we have $\lim\limits_{M\to \infty}\dfrac{A_M^{\alpha} }{ M^{\frac{\alpha}{\beta}}} =1.$  The sequence $\yn$ has PPC, so by Theorem \ref{beta_implies_alpha} it also has $(\alpha/\beta)$\,--\,PPC. Thus arguing as in \eqref{comparision} we deduce that 
\[ \lim_{M\to \infty} \frac{1}{M^{2-\frac{\alpha}{\beta}}}\#\Big\{1\ls k \neq \ell \ls M : \|y_k - y_\ell\| \ls \frac{s}{A_M^{\alpha}} \Big\} = 2s. \]
Therefore,
\begin{align*}
R_2(\alpha;s,A_M) &= \frac{\lfloor M^{\frac{1}{\beta}-1}\rfloor^2}{A_M^{2-\alpha}}\#\Big\{1\ls k \neq \ell \ls M : \|y_k - y_\ell\| \ls \frac{s}{A_M^{\alpha}} \Big\} + o(1) \\
&= \frac{\lfloor M^{\frac{1}{\beta}-1}\rfloor^2}{A_M^{2-\alpha}}  M^{2-\frac{\alpha}{\beta}} \big( 2s + o(1) \big) + o(1)\\
&= 2s + o(1), \qquad  M\to \infty.
\end{align*}
To complete the proof that $\xn$ has $\alpha$\,--\,PPC, we observe that $\lim\limits_{M\to \infty}\dfrac{A_{M+1}}{A_M}=1$
and by an approximation argument analogous to \eqref{subsequence_suffices}, we deduce that
\[ \lim\limits_{N\to \infty} R_2(\alpha; s, N) = 2s.\]


\section{Proof of Theorem \ref{weak01law} }
Theorem \ref{weak01law} will follow from a standard argument that, to the best of our knowledge, was employed for the first time in the context of pair correlations in \cite{rs} and subsequently has been used quite often in the relevant literature, see e.g. \cite{all} and the references therein. We have decided to provide only the basic elements of the proof. \par We stress out the dependence of the pair correlation function with parameter $\beta$ of the sequence $(a_nx)_{n\in\bN}$ on $x\in [0,1)$ by writing
\[R_2(\beta;s,N,x) := \frac{1}{N^{2-\beta}}\left\{1 \ls i \neq j \ls N: \norm{a_ix-a_jx} \ls \frac{s}{N^{\beta}}\right\}.\]
It then suffices to show that for Lebesgue\,--\,almost all $x \in [0,1),$ we have
\[\lim_{N \to \infty} R_2(\beta;s,N,x) = 2s\qquad \text{ for any } s>0. \]
First we fix a value of $s>0$. It follows by the definition of $R_2(\beta;s,N,x)$ that
\[\int_{0}^{1}R_2(\beta;s,N,x)\,\mathrm{d}x = \frac{2(N-1)s}{N} \cdot \]
A sufficient upper bound on the variance is given by the following lemma.
\begin{lem}
For  $\varepsilon := \frac12(1-\beta) > 0$, we have 
\[\int_{0}^{1}\left(R_2(s,N,\beta,x) - \frac{2(N-1)s}{N}\right)^2 \mathrm{d}x \ll N^{-\varepsilon}, \quad N\to \infty.\]
\end{lem}
\begin{proof}
The proof goes along the lines of \cite{all}. 
We use the Fourier series expansion
\[ \mathds{1}_{[-\tfrac{s}{N^{\beta}},\tfrac{s}{N^{\beta}}]}(x) \sim \sum_{n=-\infty}^{+\infty}c_n e(nx),  \]
where
\[ c_0 = \frac{2s}{N^{\beta}} \qquad \text{ and } \qquad c_n = \frac{1}{\pi n}\sin\Big(\frac{2\pi n s}{N^{\beta}}\Big), \quad n\neq 0.   \]
Using Parseval's identity we get
    \begin{equation}      \int_{0}^{1}\!\left(R_2(\beta;s,N,x) - \frac{2(N-1)s}{N}\right)^2\!\!\mathrm{d}x =
        \frac{1}{N^{4-2\beta}}\hspace{-4mm}\sum_{m,n \in \mathbb{Z} \setminus \{0\}}\hspace{-2mm}r_N(m)r_N(n)\hspace{-2mm}\sum_{\substack{i,j \in \mathbb{Z} \setminus \{0\}\\ im = jn}}\hspace{-4mm}c_{i}c_{j} \label{parseval}  
    \end{equation}
    where 
\[r_N(m) = \# \{1 \ls i, j \ls N, a_i - a_j = m\}, \quad m\gs 1 \]
denotes the number of representations of $m\gs 1$ as a difference of two elements of $(a_i)_{i\ls N}$. It can now be shown as in \cite{all} that for some absolute constant $C>0,$
\[ \sum_{\substack{i,j \in \mathbb{Z} \setminus \{0\}\\ im = jn}} \lvert c_{i}c_{j}\rvert   \ls C (\log N) \frac{s}{N^{\beta}} \frac{\gcd(m,n)}{\sqrt{\lvert mn \rvert}}, \quad N\gs 1. \]
Since 
\[\sum_{m \gs 1} r_N(m)^2 = \#\{i,j,k,\ell \ls N: a_i - a_j = a_k - a_\ell \} \ls N^3,\]
we can use the gcd\,--\,sum estimates from \cite[Lemma 1]{all} (in fact, the result from \cite{dyer_harman} suffices) 
to show that
\begin{equation*}\label{gcd_sum}\frac{1}{N^3}\sum_{m,n\gs 1} r_N(m) r_N(n)\frac{\text{gcd}(m,n)}{\sqrt{m n}} 
\ll \exp\left(\frac{c \sqrt{(\log N) \log \log \log N}}{\sqrt{\log \log N}}\right) \ls N^{\varepsilon}\end{equation*}
which proves the desired statement.
\end{proof}
The proof of Theorem \ref{weak01law} can be completed by arguments identical to the ones  used at the end of Theorem \ref{stronger_not_for_one}. We let  $\gamma >1/\varepsilon$ and consider the sequence $B_M := \left\lceil{M^{\gamma}}\right\rceil,\, M\gs 1.$ We  use 
Chebyshev's inequality and the first Borel\,--\,Cantelli Lemma  to deduce
that
\[\lim_{M \to \infty} R_2(\beta;s,B_M,x) = 2s\]
for almost all $x$ (where the zero\,--\,measure set  depends on $s$). For an arbitrary integer $N\gs 1$, let $M=M_N$ be the unique index such that $B_M \ls N < B_{M+1}.$ Since 
\[ \frac{B_M}{B_{M+1}}R_2\Big(\beta;\frac{B_Ms}{B_{M+1}} , B_M,x\Big) \ls R_2(\beta;s,N,x) \ls \frac{B_{M+1}}{B_M} R_2\Big(\beta;\frac{B_{M+1}s}{B_M} , B_{M+1}, x \Big)  \]
holds for any value of $x$ and also $\lim\limits_{M \to \infty} \dfrac{B_M}{B_{M+1}} = 1,$  we get 
\[\lim_{N \to \infty} R_2(\beta;s,N,x) = 2s \] for all $x$ in a set of full Lebesgue measure. It remains to show that the almost sure convergence is actually true for all $s > 0$; this follows if we consider a dense countable set in $\mathbb{R}_+$ and employ the monotonicity of $R_2(\beta;s,N,x)$ as a function of $s>0$.

\section{Proof of Theorems \ref{thm7} and \ref{higher_order_weak}}
In this last section we prove the results relevant to weak correlations of higher orders. We start with the proof Theorem \ref{higher_order_weak} and then use some of the arguments involved to deduce Theorem \ref{thm7}.  \par
Throughout this section, we shall use the following notation. Given $k \gs 2$, a  rectangle $\cR\subseteq \bR^{k-1}$  and  $0\ls \beta \ls 1$, we define a new correlation counting function by \vspace{-2mm}
\begin{equation}\label{rkstardef}
R_{k}^{*}(\beta;\mathcal{R},N)\!=\! \frac{1}{N^{k - (k-1)\beta}}
 \#\!\!\left\{i_1,\ldots,i_k\ls N\!\!: N^{\beta} ( (\!(x_{i_1}- x_{i_2})\!),\ldots, (\!(x_{i_1} - x_{i_{k-1}})\!) )\!\in\!\mathcal{R} \right\}\!.
\end{equation}
That is, in the definition of $R_k^*$ we allow indices to be equal, in contrast to the definition of $R_k$ where all indices have to be pairwise distinct.  Also whenever $a<b$ we write  
\[y_i(a,b) = y_i(a,b,N,\beta) = \#\big\{ j\ls N : a\ls N^{\beta}(\!( x_i - x_j)\!) \ls b  \big\}, \quad  i \ls N.\] 
Under this notation, we observe that for the rectangle $\cR = [a_1, b_1]\times \ldots \times [a_{k-1},b_{k-1}]$ we have
\begin{equation}\label{R_k_easy_form}R_k^*(\beta;\mathcal{R},N) = \frac{1}{N^{k-(k-1)\beta}}\sum_{i\ls N} y_i(a_1,b_1)\cdot  \ldots \cdot  y_i(a_{k-1},b_{k-1}).\end{equation}
As we shall soon see, this form makes $R_k^*$ much easier to handle than $R_k$. 

The following lemma states that for $\beta < 1$, under the assumption of Poissonian $(k,\beta)$\,--\,correlations, the asymptotic size of $R_k$ is the same as that of $R_k^{*}$ as $N \to \infty$. In the proof of the lemma we make use of certain facts on $R_k$ and $R_k^*$ that appear in  \cite{our_other_paper}. There are only minor modifications, and these are due to the fact that here we deal with weak correlations; we have chosen to refer the interested reader to the proofs in \cite{our_other_paper} and explain here briefly where the minor differences come from.

\begin{lem}\label{equivalent_k_correl}
Let $\xn$ be a sequence and $k \gs 2$. The following are equivalent.\newline   {\it (i)} $\xn$ has Poissonian $(k,\beta)$\,--\,correlations. \newline
   {\it (ii)} For all closed rectangles $\mathcal{R} \subseteq \mathbb{R}^{k-1}$ of the form $[a_1,b_1]\times \ldots \times [a_{k-1},b_{k-1}]$ we have 
    \begin{equation}\label{higher_order_alter}\lim_{N \to \infty} R_k^{*}(\beta;\mathcal{R},N) = \lambda(\mathcal{R}).\end{equation}
  {\it (iii)} For all cubes $\mathcal{C} := [a,b]^{k-1}, a,b \in \mathbb{R}$ we have
  \begin{equation}\label{smaller_equal_suffices}\limsup_{N \to \infty} R_{k}^{*}(\beta;\mathcal{C},N) \ls \lambda(\mathcal{C}).\end{equation}
\end{lem}

\begin{proof} We start by proving the equivalence of (i) and (ii). 
We observe that for any rectangle $\mathcal{R}\subseteq \mathbb{R}^{k-1},$ we have 
\begin{equation}\label{compare_R_k*}R_{k}(\beta;\cR,N) \ls R_{k}^*(\beta;\cR,N) \ls R_{k}(\beta;\cR,N) + \sum_{m= 1}^{k-1}\Big(\frac{1}{N^{1-\beta}}\Big)^{k-m}b_mR_{m}(\beta;\cR_m,N)
\end{equation}
where $R_1(s;N) := 1$, for each $1\ls m \ls k-1, \cR_{m}\subseteq \mathbb{R}^{m-1}$ is an $(m-1)$\,--\,dimensional rectangle depending on $\cR,$ and $b_1,\ldots, b_{k-1} \in \mathbb{N}$ are constants depending only on $k$. The proof of \eqref{compare_R_k*} is essentially the same as that of \cite[Proposition 7]{our_other_paper}, which deals with standard correlation functions (i.e. when $\beta = 1$). The only differences are the appearance of the coefficients $(N^{-(1-\beta)})^{k-m}$, that come from the different scaling factor for weak correlations, and the fact that \eqref{compare_R_k*} is about arbitrary  rectangles $\cR$ rather than rectangles symmetric with respect to the axes.

 Next, we observe that whenever ${\rm (i)}$ or ${\rm (ii)}$ holds, arguing  analogously to \cite[Theorem 3]{our_other_paper}, we can show that for $m \ls k-1$ we have
\begin{equation}\label{finite_limsup}\limsup_{N\to\infty}R_{m}(\beta;\mathcal{S},N) < \infty \quad \text{ for any rectangle } \mathcal{S}\subset \mathbb{R}^{m-1}. \end{equation}
Combining \eqref{compare_R_k*} and \eqref{finite_limsup} completes the proof of this equivalence.

Since (ii) $\Rightarrow$ (iii) is trival, it remains to show (iii) $\Rightarrow$ (ii).
So let us assume that \eqref{smaller_equal_suffices} holds and let $ \cR = [a_1,b_1]\times\ldots\times [a_{k-1},b_{k-1}] $ be an arbitrary rectangle in $\mathbb{R}^{k-1}$.
Applying the Hölder inequality with exponents $p_i = 1/(k-1), ( 1\ls i \ls k-1)$ to \eqref{R_k_easy_form} we obtain
\begin{align}
R_{k}^{*}(\beta;\mathcal{R},N) &= \frac{1}{N^{k-(k-1)\beta}}\sum_{i\ls N} y_i(a_1,b_1)\cdot \ldots \cdot y_i(a_{k-1},b_{k-1})\\
&\ls \frac{1}{N^{k-(k-1)\beta}}\Bigg(\sum_{i\ls N} y_i(a_1,b_1)^{k-1}\Bigg)^{\frac{1}{k-1}}\hspace{-2mm}\cdot \ldots \cdot  \Bigg(\sum_{i\ls N} y_i(a_{k-1},b_{k-1})^{k-1}\Bigg)^{\frac{1}{k-1}} \label{hölder_on_yi}\\
&= R_{k}^{*}(\beta;\mathcal{C}_{a_1,b_1},N)^{\frac{1}{k-1}} \cdot \ldots \cdot  R_{k}^{*}(\beta;\mathcal{C}_{a_{k-1},b_{k-1}},N)^{\frac{1}{k-1}},
\end{align}
where we set $\mathcal{C}_{a_i,b_i}=[a_i, b_i]^{k-1}$ for each $1\ls i \ls k-1.$ Using \eqref{smaller_equal_suffices} we are able to deduce that
\begin{equation}\label{smaller_than}\limsup_{N \to \infty} R_{k}^{*}(\beta;\mathcal{R},N)  \ls \lambda(\mathcal{R})\quad \text{ for any rectangle } \mathcal{R}\subset \mathbb{R}^{k-1}.\end{equation}
For the remaining part of the proof, it will be convenient to adopt the following notation: given any $\bt = (t_1,\ldots, t_{k-1})\in \mathbb{R}^{k-1}$ with $t_i > 0, i=1,\ldots, k-1$ we shall write 
\[\cR_{\bt} = [-t_1, t_1] \times \ldots \times [-t_{k-1}, t_{k-1}] \subseteq \mathbb{R}^{k-1} \]
for the rectangle in $\mathbb{R}^{k-1}$ that is symmetric with respect to each of the axes and has a vertex at the point $\bt.$ \par Continuing with the proof of the implication (iii) $\Rightarrow$ (ii), we assume for contradiction that \eqref{higher_order_alter} does not hold. Then by \eqref{smaller_than} there must exist \mbox{$\ba = (a_1,\ldots,a_{k-1}),$} \mbox{$\bb=(b_1,\ldots,b_{k-1}) \in \mathbb{R}^{k-1}$} such that for the rectangle $\cR_{[\ba,\bb]}:=[a_1,b_1]\times \ldots\times [a_{k-1},b_{k-1}]$ we have 
 \begin{equation}\label{to_contradict}
 \liminf_{N \to \infty}R_{k}^{*}(\beta;\mathcal{R}_{[\ba,\bb]},N) <  \lambda(\mathcal{R}_{[\ba,\bb]}).\end{equation}
Setting $s  = \max\{\vert a_i\vert, \vert b_i\vert  : i=1,\ldots, k-1 \}$ and $\bs = (s, \ldots, s)$, we can find $n \in \mathbb{N}$ and rectangles $\mathcal{R}_{j}, 1\ls j\ls  n$  such that the cube $\mathcal{C}_s  := [-s, s]^{k-1}$ satisfies
\[ \mathcal{C}_s = \mathcal{R}_{[\ba,\bb]} \cup \bigcup_{j = 1}^n \mathcal{R}_{j} \]
and all rectangles  in the right-hand side have disjoint interiors. Using \eqref{smaller_than} on every $\mathcal{R}_j$ and \eqref{to_contradict}, we obtain
\begin{align}
\liminf_{N \to \infty}R_k^{*}(\beta; \mathcal{C}_s ,N)  &\ls  \liminf_{N \to \infty}R_{k}^{*}(\beta;\mathcal{R}_{[\ba,\bb]},N)+
    \sum_{j =1}^n \limsup_{N \to \infty} R_{k}^{*}(\beta;\mathcal{R}_{j},N)
    \\&< (2s)^{k-1}.\label{partitioned_liminf}
\end{align}
Consequently, there exist an $\eta>0$ and a sequence $(N_r)_{r \in \mathbb{N}}\subseteq \mathbb{N}$ such that 
\begin{equation}
    R_k^{*}(\beta; \mathcal{C}_s ,N_r) < (2s)^{k-1} - \eta \qquad \text{for all } r\gs 1.
\end{equation}
 Observe that  there exists a cube $\mathcal{C}_0 = [s-\delta, s]\times \ldots \times [s- \delta, s]  \subseteq \mathbb{R}^{k-1}$ with positive $(k-1)$\,--\,dimensional Lebesgue measure such that for any $r\gs 1,$
  \begin{equation}\label{strictly_smaller}R_{k}^{*}(\beta;\cR_\bt,N_r) < (2t_1)(2t_2)\cdots(2t_{k-1})- \frac{\eta}{2}\end{equation}
for any $\bt = (t_1,\ldots, t_{k-1}) \in \mathcal{C}_0$. \par
 We shall obtain the desired contradiction by deriving a lower and an upper bound for the integral \[ \int_{[0,s]^{k-1}} R_k^{*}(\beta,\cR_\bt\,,N) \,\mathrm{d}\bt. \] Arguing as in \cite[Proposition 11]{our_other_paper}, we can show that
 \begin{equation}\label{prop11}\int_{[0,s]^{k-1}} R_k^{*}(\beta,\cR_\bt\,,N) \,\mathrm{d}\bt \gs s^{2(k-1)} \quad \text{ for any } N \gs 1.\end{equation}
Let $\varepsilon > 0$ be arbitrary. For any integer $M  \in \mathbb{N},$ we can split the integration domain $[0,s]^{k-1}$ in $M^{k-1}$ cubes of the form
\[\mathcal{C}_{\bj} := \left[\frac{(j_1-1) s}{M},\frac{j_1 s}{M}\right]\times 
\ldots  \times \left[\frac{(j_{k-1}-1) s}{M},\frac{j_{k-1}s}{M}\right],  \]
where $\bj = (j_1, \ldots, j_{k-1}) \in [M]^{k-1}$, and if $M\gs 1$ is large enough, then for all indices $\bj \in [M]^{k-1}$ we have 
\begin{equation}\label{small_diam}
\lvert \lambda(\mathcal{R}_{\bt_1}) - \lambda(\mathcal{R}_{\bt_2})\rvert \ls  \frac{\varepsilon}{2} \qquad \text{ for any } \bt_1, \bt_2 \in \mathcal{C}_{\bj}.
\end{equation}
For each of the cubes $\mathcal{C}_{\bj}$ defined above, we denote its top corner by  $$\bc_{\bj} =  \left(\frac{j_1s}{M},\frac{j_2s}{M},\ldots, \frac{j_{k-1}s}{M}\right) \in \mathbb{R}^{k-1}.$$  In view of \eqref{smaller_than}, we know that for all $r\gs 1$ sufficiently large, we have
\begin{equation} \label{corner_bound}
    R_k^{*}(\beta;\mathcal{R}_{\bc_\bj},N_r) \ls \lambda(\mathcal{R}_{\bc_\bj}) + \frac{\varepsilon}{2} \qquad \text{for all } \bj\in [M]^{k-1}.
\end{equation}
(Note that the values of $r$ can be chosen independently of $\bj$ since the number of considered $\bj$'s is finite.) Given $\bt \in [0,s]^{k-1}$ we write $\mathbf{c}(\bt)$ for the top corner $\bc_{\bj}$ of the cube $\mathcal{C}_{\bj}$ that contains the point $\bt.$ Then for all $r\gs 1$ large enough, we have 
\begin{align}R_k^{*}(\beta,\cR_\bt,N_r)
&\ls R_k^{*}(\beta,\cR_{\bc(\bt)},N_r) \\ & \ls \lambda(\mathcal{R}_{\bc_\bj}) + \frac{\varepsilon}{2} \quad \qquad \text{ (by \eqref{corner_bound})}\label{all_small_enough} \\[1ex]  & \ls \lambda(\mathcal{R}_{\bt}) + \varepsilon \quad \qquad \, \, \, \text{ (by \eqref{small_diam})}  \end{align}
for any $\bt \in [0,s]^{k-1}$.
Combining \eqref{strictly_smaller} and \eqref{all_small_enough}, we obtain
\begin{align}
\int_{[0,s]^{k-1}} R_k^{*}(\beta,\cR_\bt\,,N_r) \,\mathrm{d}\bt &= \int_{[0,s]^{k-1} \setminus \mathcal{C}_0} R_k^{*}(\beta;\cR_\bt\,,N_r) \,\mathrm{d}\bt 
+ \int_{\mathcal{C}_0} R_k^{*}(\beta;\cR_\bt\,,N_r) \,\mathrm{d}\bt 
\\&\ls s^{2(k-1)} + \varepsilon s^{k-1} - \eta\lambda(\mathcal{C}_0),
\end{align}
a contradiction to \eqref{prop11} when $\varepsilon$ is  sufficiently small.
 \end{proof}
 
Having established Lemma \ref{equivalent_k_correl}, we can now proceed to the proof of Theorem \ref{higher_order_weak}.

\begin{proof}[Proof of Theorem \ref{higher_order_weak}]
    We first prove that having Poissonian $(k,\beta)$\,--\,correlations is a property stronger than $\beta$\,--\,PPC. We then use this fact to prove that this is also stronger than $(k-1, \beta)$\,--\,correlations. We apply the Hölder inequality with exponents \mbox{$p = (k-1)/(k-2)$} and $q = k-1$ to obtain
    \begin{align}
    R_2^{*}(\beta;[a,b],N)^{k-1} &= \frac{1}{N^{2(k-1)-(k-1)\beta}}\Big(\sum_{i\ls N} y_i(a,b)\Big)^{k-1} \nonumber \\
    & \ls \frac{1}{N^{k-(k-1)\beta}}\sum_{i\ls N} y_i(a,b)^{k-1} = R_k^{*}(\beta;[a,b]^{k-1},N).  \label{hölder}
    \end{align}
    Assuming Poissonian $(k,\beta)$\,--\,correlations, the right-hand side of \eqref{hölder} tends to \newline $(b-a)^{k-1}$ as $N \to \infty$. Therefore, we have for all $a < b \in \mathbb{R}$ that
    \[\limsup_{N \to \infty} R_2^{*}(\beta;[a,b],N) \ls b-a,\]
    which  by Lemma \ref{equivalent_k_correl} implies $\beta$\,--\,PPC.

    Now we proceed to prove that $\xn$ also has Poissonian $(k-1,\beta)$\,--\,correlations.
    We make use of the well\,--\,known inequality
    \begin{equation}\label{chebyshev_cor}
        N\sum_{i\ls N} x_i^{k-1}   \, \gs\,   \Big(\sum_{i\ls N} x_i\Big)\Big(\sum_{i\ls N} x_i^{k-2} \Big)
    \end{equation}
    which holds for $N \gs 1, x_i \gs 0$ to show that
    \begin{align}R_k^*(\beta;[a,b]^{k-1},N) &= \frac{1}{N^{k-(k-1)\beta}}\sum_{i\ls N} y_i(a,b)^{k-1}
    \\&\gs \frac{1}{N^{2-\beta}}
    \sum_{i\ls N} y_i(a,b) \cdot \frac{1}{N^{k-1-(k-2)\beta}}\sum_{i\ls N} y_i(a,b)^{k-2}
\\   & = R_2^{*}(\beta;[a,b],N)\cdot R_{k-1}^*(\beta;[a,b]^{k-2},N).
    \end{align}
By Poissonian $(k,\beta)$\,--\,correlations and $\beta$\,--\,PPC, we can deduce that
\[\limsup_{N \to \infty} R_{k-1}^*(\beta;[a,b]^{k-2},N) \ls (b-a)^{k-2},\]
and in view of Lemma \ref{equivalent_k_correl} this completes the proof. 
 \end{proof}
 
\begin{xrem}
Note that Theorem \ref{higher_order_weak} holds for $\beta = 0$, the proof follows the same lines as above with the only difference being $s < \frac{1}{2}$. However for $\beta = 1$, this argumentation fails since in case of Poissonian $k$\,--\,correlations, $R_k^*(\mathcal{R},N)$ will not tend to $\lambda(\mathcal{R})$.
\end{xrem}

 \begin{proof}[Proof of Theorem \ref{thm7}]
The one direction of the statement, namely that Poissonian $(k,0)$\,--\,correlations imply uniform distribution, follows from  Theorem \ref{higher_order_weak}: for $\beta = 0$, we have proven that Poissonian $(k,0)$-correlations imply $0$\,--\,PPC, which in turn imply uniform distribution
by Theorem \ref{prop1}. \par 
We proceed to showing that uniform distribution implies Poissonian $(k,0)$\,--\,correlations.
In view of Lemma \ref{equivalent_k_correl}, it suffices to prove that for any cube $ \mathcal{C} = [a,b]^{k-1} \subseteq   [-\tfrac{1}{2},\tfrac{1}{2} ]^{k-1} $ we have
\begin{equation}\label{suff_thm7}\limsup_{N \to \infty} R_k^{*}(0;\mathcal{C} ,N)  \ls \lambda(\mathcal{C} ).\end{equation}
As expected, the proof goes along the lines of Theorem \ref{prop1}. We observe that
\begin{align*}
    R_k^{*}(0; \mathcal{C},N) - \lambda(\mathcal{C})  & \, =\, \frac{1}{N^k}\hspace{-3mm}\sum_{i_1,\ldots, i_k \ls N}\hspace{-3mm}\one_{[a,b]}(x_{i_1}-x_{i_2}) \cdots \one_{[a,b]}(x_{i_1}-x_{i_{k-1}})  - \lambda(\mathcal{C}) \\
    &=   \frac{1}{N}\sum_{n\ls N}\Bigg(\frac{1}{N}\sum_{m\ls N} \one_{[x_{n}-b,  x_{n}-a]} (x_{m})\Bigg)^{k-1}- (b-a)^{k-1}.
\end{align*}
Since 
\begin{align*} 
\frac{1}{N}\sum_{m\ls N} \one_{[x_{n}-b,  x_{n}-a]} (x_{m})
\ls b-a + D_N,\quad  n = 1,\ldots, N, 
\end{align*}
with $D_N$ denoting the discrepancy of $\xn$, we deduce that
\[ R_k^{*}(0;\mathcal{C} ,N) - \lambda(\mathcal{C} )
\ls (b-a + D_N)^{k-1} - (b-a)^{k-1} =  \mathcal{O}_{a,b}( D_N ).\]
Assuming that $\xn$ is uniformly distributed mod $1$, we have $D_N  \to  0$ as $N\to \infty $ and hence \eqref{suff_thm7} follows. This concludes the proof.\end{proof}

\subsection*{Acknowledgements} We would like to thank Niclas Technau for introducing us to the notion of weak Poissonian pair correlations. We also thank Christoph Aistleitner for many valuable discussions.

\end{document}